\documentclass[reqno]{amsart}
 \newtheorem{theorem}{Theorem}[section]
 \newtheorem{corollary}[theorem]{Corollary}
 \newtheorem{lemma}{Lemma}[section]
 \newtheorem{proposition}[theorem]{Proposition}
 \theoremstyle{definition}
 
 \theoremstyle{remark}
 \newtheorem{remark}{Remark}[section]
 \theoremstyle{eg}

 \theoremstyle{fact}
 
\numberwithin{equation}{section}

\begin{document}
\title[generalized CQ method]
{Strong convergence theorems by generalized CQ method in Hilbert
spaces}

\author[S. He]{Songnian He$^*$ }
\thanks{$^*$Supported by Tianjin Natural Science Foundation in China
Grant (06YFJMJC12500). }

\address{$^*$College of Since, Civil Aviation University of China,
Tianjin 300300, China} \email{hesongnian2003@yahoo.com.cn}

\author[T. Shi]{Tian Shi$^{**}$}
\address{$^{**}$College of Since, Civil Aviation University of China,
Tianjin 300300, China} \email{tianapply@gmail.com}
\thanks{$^{**}$Corresponding author.}



\maketitle

\begin{abstract}
Recently, CQ method has been investigated extensively. However, it
is mainly applied to modify Mann, Ishikawa and Halpern iterations to
get strong convergence. In this paper, we study the properties of CQ
method and proposed a framework. Based on that, we obtain a series
of strong convergence theorems. Some of them are the extensions of
previous results. On the other hand, CQ method, monotone Q method,
monotone C method and monotone CQ method, used to be given
separately, have the following relations: CQ method TRUE
$\Rightarrow$ monotone Q method TRUE $\Rightarrow$ monotone C method
TRUE $\Leftrightarrow$ monotone CQ method TRUE.

\end{abstract}

\date{2009-12-18}

\section{Introduction}
Let $C$ be a nonempty closed convex subset of a real Hilbert space
$H$ and $T$ a self-mapping of $C$. Recall that $T$ is said to be a
nonexpansive mapping if
\begin{equation}
\|Tx-Ty\|\leq\|x-y\|, \forall x,y\in C.
\end{equation}
$T$ is said to be strictly pseudo-contractive if there exists a
constant $0\leq\kappa<1$ such that
\begin{equation}
\|Tx-Ty\|^2\leq\|x-y\|^2+\kappa\|(I-T)x-(I-T)y\|^2
\end{equation}
for all $x,y\in C$. For such cases, $T$ is also said to be a
$\kappa$-strictly pseudo-contractive mapping. It is also said to be
pseudo-contractive if $\kappa=1$ in (1.2). That is,
\begin{equation}
\|Tx-Ty\|^2\leq\|x-y\|^2+\|(I-T)x-(I-T)y\|^2
\end{equation}
for all $x,y\in C$. Clearly, the class of strict pseudo-contractions
falls into the one between classes of nonexpansive mappings and
pseudo-contractions.

It is very clear that, in a real Hilbert space $H$, (1.3) is
equivalent to
\begin{equation}
\langle Tx-Ty,x-y\rangle\leq\|x-y\|^2
\end{equation}

Recall that three iteration processes are often used to approximate
a fixed point of a nonexpansive mapping. The first one is Halpern's
iteration process \cite{2} which is defined as follows: Take an
initial guess $x_0\in C$ arbitrarily and define $\{x_n\}$
recursively by
\begin{equation}
x_{n+1}=(1-\alpha_n) x_0+\alpha_n Tx_n, \ n\geq 0
\end{equation}
where $\{\alpha_n\}$ is a sequence in the interval $[0,1]$.

The second is known as Mann's iteration process \cite{4} which is
defined as
\begin{equation}
x_{n+1}=(1-\alpha_n) x_n+\alpha_n Tx_n, \ n\geq 0
\end{equation}
where the initial guess $x_0$ is taken in $C$ arbitrarily and the
sequence $\{\alpha_n\}$ is in the interval $[0,1]$.

The third is referred to as Ishikawa's iteration process \cite{3}
which is defined recursively by
\begin{equation}
\left\{\begin{array}{l}z_n=(1-\alpha_n) x_n+\alpha_n Tx_n\\
x_{n+1}=(1-\beta_n) x_n+\beta_nTz_n
\end{array}\right.
, \ n\geq 0
\end{equation}
where the initial guess $x_0$ is taken in $C$ arbitrarily and
$\{\alpha_n\}$ and $\{\beta_n\}$ are sequences in the interval
$[0,1]$.

We know that (1.5) has strong convergence under certain conditions,
but both (1.6) and (1.7) have only weak convergence, in general,
even for nonexpansive mappings (see an example in \cite{1}).

Recently, modifications of algorithm (1.5), (1.6) and (1.7) have
been extensively investigated; see \cite{5,6,8,9,10,11} and the
references therein. For instance, one of the most important methods
was firstly introduced by Nakajo and Takahashi \cite{6} in 2003.

\begin{theorem}[see \cite{6}]
Let $C$ be a nonempty closed convex subset of a Hilbert space $H$
and $T$ a nonexpansive mapping of \ $C$ into itself such that
$F(T)\neq \O$. Suppose $x_0\in C$ chosen arbitrarily and $\{x_n\}$
is given by
\begin{equation}
\left\{\begin{array}{l}y_n=(1-\alpha_n) x_n+\alpha_n Tx_n\\
C_n=\{z\in C: \|z-y_n\|\leq\|z-x_n\|\}\\
Q_n=\{z\in C: \langle z-x_n,x_n-x_0\rangle\geq 0\}\\
x_{n+1}=P_{C_n\cap Q_n}x_0
\end{array}\right.
\end{equation}
where $P_{C_n\cap Q_n}$ is the metric projection from $C$ onto
$C_n\cap Q_n$ and $\{\alpha_n\}$ is chosen such that
$0<\alpha\leq\alpha_n\leq 1$. Then, $\{x_n\}$ converges strongly to
$P_{F(T)}x_0$, where $P_{F(T)}$ is the metric projection from $C$
onto $F(T)$.
\end{theorem}

\begin{remark}
It is also known as CQ method or CQ method. The purpose of the
authors is to modify Mann's iteration process and obtain a strong
convergent sequence. However, we can learn more from (1.8). In fact,
(1.8) is equivalent to

\begin{equation}
\left\{\begin{array}{l}C_n=\{z\in C: \|z-((1-\alpha_n) x_n+\alpha_n Tx_n)\|\leq\|z-x_n\|\}\\
Q_n=\{z\in C: \langle z-x_n,x_n-x_0\rangle\geq 0\}\\
x_{n+1}=P_{C_n\cap Q_n}x_0
\end{array}\right.
\end{equation}

From (1.9) we can conclude that in each recursive step, the
algorithms can be divided into two parts:

($P_1$) construct an appropriate set;

($P_2$) project the given fixed point onto the set.

According to this point view, the crux of CQ method is how to
construct an appropriate set and (1.9) is just a special case:

($A_1$) construct $C_n$ based on iteration scheme (1.6) and the
properties of the mapping $T$.

($A_2$) construct $Q_n$ by the property of the metric projection.

$C_n\cap Q_n$ is the appropriate set. Then, together with ($P_2$),
we can yield (1.9), i.e., Theorem 1.1.

Actually, based on this idea we can accomplish (P1) in many ways and
construct different kinds of appropriate sets based on scheme (1.5),
(1.6), (1.7) and their combinations. And we name this method as {\it
generalized CQ method}.
\end{remark}

Motivated by Remark 1.1, we propose a CQ algorithm framework, which
is the basic work in this paper. Then, based on this framework, we
introduce a series of strong convergence theorems. Some of them,
used to be given separately, have direct relations between each
other.

In section 8, we study the relations among CQ method, monotone Q
method, monotone C method and monotone CQ method.

\section{Preliminaries and lemmas}

Let $H$ be a real Hilbert space with inner product
$\langle\cdot,\cdot\rangle$ and let $C$ be a closed convex subset of
$H$. For every point $x\in H$ there exists a unique nearest point in
$C$, denoted by $P_Cx$ such that
$$
\|x-P_Cx\|\leq\|x-y\|
$$
for all $y\in C$, where $P_C$ is called the metric projection of $H$
onto $C$. We know that $P_C$ is a nonexpansive mapping.

$x_n\rightarrow x$ implies that $\{x_n\}$ converges strongly to $x$.
$x_n\rightharpoonup x$ means ${x_n}$ converges weakly to $x$.

We know that a Hilbert space $H$ satisfies Opial's condition
\cite{7}, that is, for any sequence $\{x_n\}\subset H$ with
$x_n\rightharpoonup x$, the inequality
$$
\liminf_{n\rightarrow\infty}\|x_n-x\|<\liminf_{n\rightarrow\infty}\|x_n-y\|
$$
holds for every $y\in H$ with $y\neq x$. We also know that $H$ has
the Kadec-Klee property, that is $x_n\rightharpoonup x$ and
$\|x_n\|\rightarrow\|x\|$ imply $x_n\rightarrow x$. In fact, from
$\|x_n-x\|^2=\|x_n\|^2-2\langle x_n,x\rangle+\|x\|^2$, we get that a
Hilbert space has Kadec-Klee property.

For a given sequence $\{x_n\}\subset C$, let $\omega_w(x_n)=\{x:
\exists x_{n_j}\rightharpoonup x\}$ denote the weak limit set of
$\{x_n\}$

Now we collect some lemmas which will be used in the proof of main
theorems in following sections.

\begin{lemma}[see \cite{5}]
Let $H$ be a real Hilbert space. There hold the following
identities:\\
(i) $\|x-y\|^2=\|x\|^2-\|y\|^2-2\langle x-y,y\rangle, \ \forall
x,y\in H$\\
(ii) $\|\alpha
x+(1-\alpha)y\|^2=\alpha\|x\|^2+(1-\alpha)\|y\|^2-\alpha(1-\alpha)\|x-y\|^2,
\ \forall\alpha\in [0,1] \ and \ x,y\in H$.
\end{lemma}

\begin{lemma}
Let $C$ be a closed convex subset of real Hilbert space $H$. Given
$x\in H$ and $z\in C$. Then $z=P_C x$ if and only if there holds the
relation
$$
\langle x-z,y-z\rangle\leq 0
$$
for all $y\in C$.
\end{lemma}

\begin{lemma}[see \cite{9}]
Let $H$ be a real Hilbert space. Given a closed convex subset $C$
and points $x,y,z\in H$. Given also a real number $a\in\mathbb{R}$.
The set
$$
\{v\in C:\|y-v\|^2\leq\|x-v\|^2+\langle z,v\rangle+a\}
$$
is closed and convex.
\end{lemma}

\begin{lemma}[see \cite{9}]
Let $C$ be a closed convex subset of real Hilbert space $H$. Let
$\{x_n\}$ be a sequence in $H$ and $u\in H$. Let $q=P_C u$. If
$\omega_w(x_n)\subset C$ and
$$
\|x_n-u\|\leq\|u-q\|
$$
for all $n$, then $x_n\rightarrow q$.
\end{lemma}

\begin{lemma}[see \cite{11}]
Let $C$ be a nonempty closed convex subset of a real Hilbert space
$H$ and $T:C\rightarrow C$ a demi-continuous pseudo-contractive
self-mapping from $C$ into itself. Then $F(T)$ is a closed convex
subset of $C$ and $I-T$ is demiclosed at zero.
\end{lemma}

\begin{lemma}[see \cite{5}]
Let $C$ be a nonempty closed convex subset of $H$ and
$T:C\rightarrow C$ a $\kappa-$strict pseudo-contraction for some
$0\leq\kappa<1$. Then $F(T)$ is a closed convex subset of $C$ and
$I-T$ is demiclosed at zero.
\end{lemma}

\begin{lemma}
Let $C$ be a nonempty closed convex subset of a real Hilbert space
$H$ and $T$ an $L-$Lipschitz mapping from $C$ into itself. Assume
$F(T)\neq\O$ is closed and convex, $L+1\leq\mu<\infty$ and
$\theta\geq 0$. Let
$$
C_x=\{z\in C: \|x-Tx\|\leq\mu\|x-z\|+\theta\}, \ \forall x\in C.
$$
Then, $F(T)\subset C_x$.
\end{lemma}

\begin{proof}
Let $p\in F(T)$, we have $\forall x\in C$
$$
\aligned \|x-Tx\|&\leq\|x-p\|+\|p-Tx\|\\
&\leq\|x-p\|+L\|x-p\|\\
&=(L+1)\|x-p\|\\
&\leq\mu\|x-p\|+\theta.
\endaligned
$$
Hence, $p\in C_x$, i.e., $F(T)\subset C_x$.

\end{proof}

\section{Main result}

In this section, a strong convergence theorem is obtained by
generalized CQ method.

Using Lemma 2.7, we can introduce the following theorem.

\begin{theorem}
Let $C$ be a nonempty closed convex subset of a real Hilbert space
$H$ and $T$ an $L-$Lipschitz mapping from $C$ into itself. Assume
$F(T)\neq\O$ is closed and convex, $I-T$ is demiclosed at zero,
$\{\mu_n\}$ is a sequence such that $L+1\leq\mu_n\leq\mu<\infty$ and
$\theta_n(z)$ is a nonnegative function on $C$. Let $\{x_n\}$ be a
sequence generated by the following manner:
\begin{equation}
\left\{\begin{array}{l}x_0\in C \ chosen \ arbitrarily\\
C_n=\{z\in C: \|x_n-T x_n\|\leq\mu_n\|x_n-z\|+\theta_n(z)\}\\
Q_n=\{z\in C: \langle z-x_n,x_n-x_0\rangle\geq 0\}\\
x_{n+1}=P_{C_n^*\cap Q_n}x_0
\end{array}\right.
\end{equation}
where $C_n^*$ is a closed convex set with $F(T)\subset C_n^*\subset
C_n$. Assume $\lim_{n\rightarrow\infty}\theta_n(x_{n+1})=0$. Then
$\{x_n\}$ converges strongly to $P_{F(T)}x_0$.
\end{theorem}

\begin{proof}
According to the assumption, we see that $P_{F(T)}x_0$ is well
defined. It is obvious that $Q_n$ is closed and convex, hence,
$C_n^*\cap Q_n$ is closed and convex.

Next, we show that $F(T)\subset C_n^*\cap Q_n$. From the assumption,
$F(T)\subset C_n^*$, hence, it suffices to prove $F(T)\subset Q_n$.
We prove this by induction. For $n=0$, we have $F(T)\subset C=Q_0$.
Assume that $F(T)\subset Q_n$. Since $x_{n+1}$ is the projection of
$x_0$ onto $C_n^*\cap Q_n$, we have
$$
\langle z-x_{n+1},x_{n+1}-x_0\rangle\geq 0, \forall z\in C_n^*\cap
Q_n.
$$
As $F(T)\subset C_n^*\cap Q_n$ by the induction assumption, the last
inequality holds, in particular, for all $z\in F(T)$. This together
with the definition of $Q_{n+1}$ implies that $F(T)\subset Q_{n+1}$.
Hence, $F(T)\subset Q_n$ holds for all $n\geq 0$ and $\{x_n\}$ is
well defined.

From $x_n=P_{Q_n}x_0$, we have
$$
\langle x_0-x_n,x_n-y\rangle\geq 0
$$
for all $y\in C_n^*\cap Q_n$. So, for $p\in F(T)$, we have
$$
\aligned 0&\leq\langle x_0-x_n,x_n-p\rangle\\
&=\langle x_0-x_n,x_n-x_0+x_0-p\rangle\\
&=-\|x_0-x_n\|^2+\langle x_0-x_n,x_0-p\rangle\\
&\leq-\|x_0-x_n\|^2+\|x_0-x_n\|\cdot\|x_0-p\|.
\endaligned
$$
Hence,
\begin{equation}
\|x_0-x_n\|\leq\|x_0-p\|
\end{equation}
for all $p\in F(T)$. This implies that $\{x_n\}$ is bounded.

From $x_n=P_{Q_n}x_0$ and $x_{n+1}=P_{C_n^*\cap Q_n}x_0\in C_n^*\cap
Q_n$, we have
$$
\langle x_0-x_n,x_n-x_{n+1}\rangle\geq 0.
$$
Hence,
$$
\aligned 0&\leq\langle x_0-x_n,x_n-x_{n+1}\rangle\\
&=\langle x_0-x_n,x_n-x_0+x_0-x_{n+1}\rangle\\
&=-\|x_0-x_n\|^2+\langle x_0-x_n,x_0-x_{n+1}\rangle\\
&\leq-\|x_0-x_n\|^2+\|x_0-x_n\|\cdot\|x_0-x_{n+1}\|,
\endaligned
$$
therefore
$$
\|x_0-x_n\|\leq\|x_0-x_{n+1}\|,
$$
which implies that $\lim_{n\rightarrow\infty}\|x_n-x_0\|$ exists.

Besides, by Lemma 2.1 we have
$$
\aligned \|x_{n+1}-x_n\|^2&=\|(x_{n+1}-x_0)-(x_n-x_0)\|^2\\
&=\|x_{n+1}-x_0\|^2-\|x_n-x_0\|^2-2\langle x_{n+1}-x_n,x_n-x_0\rangle\\
&\leq\|x_{n+1}-x_0\|^2-\|x_n-x_0\|^2.
\endaligned
$$
Let $n\rightarrow\infty$, we get
$\lim_{n\rightarrow\infty}\|x_{n+1}-x_n\|=0$.

Noticing $x_{n+1}=P_{C_n^*\cap Q_n}x_0\subset C_n$, we have
$$
\|x_n-T x_n\|\leq\mu_n\|x_n-x_{n+1}\|+\theta_n(x_{n+1}).
$$
Combining with the assumption of $\{\mu_n\}$ and
$\{\theta_n(x_{n+1})\}$, we obtain
$$
\lim_{n\rightarrow\infty}\|x_n-T x_n\|=0.
$$
Since $I-T$ is demiclosed, then every weak limit point of $\{x_n\}$
is a fixed point of $T$. That is, $\omega_w{(x_n)}\subset F(T)$. By
Lemma 2.4, $x_n\rightarrow P_{F(T)}x_0$.
\end{proof}

\section{Applications of the main result}

First, we use Theorem 3.1 to prove Theorem 1.1.

Obviously, the following theorem can be easily verified by Theorem
3.1.

\begin{theorem}
Let $C$ be a nonempty closed convex subset of a Hilbert space $H$
and $T$ a nonexpansive mapping of \ $C$ into itself such that
$F(T)\neq \O$. Suppose $x_0\in C$ chosen arbitrarily and $\{x_n\}$
is given by
$$
\left\{\begin{array}{l}y_n=(1-\alpha_n) x_n+\alpha_n Tx_n\\
{^*C_n}=\{z\in C:\|x_n-Tx_n\|\leq\frac{2}{\alpha_n}\|x_n-z\|\}\\
Q_n=\{z\in C: \langle z-x_n,x_n-x_0\rangle\geq 0\}\\
x_{n+1}=P_{C_n^{**}\cap Q_n}x_0
\end{array}\right.
$$
where $C_n^{**}$ is a closed convex set with $F(T)\subset
C_n^{**}\subset {^*C_n}$ and $\{\alpha_n\}$ is chosen such that
$0<\alpha\leq\alpha_n\leq 1$. Then, $\{x_n\}$ converges strongly to
$P_{F(T)}x_0$.
\end{theorem}

Let $C_n^{**}=C_n=\{z\in C:\|y_n-z\|\leq\|x_n-z\|\}$ in Theorem 4.1.
Easily, we can prove $C_n$ is closed and convex with $F(T)\subset
C_n\subset {^*C_n}$. So, Theorem 4.2 is valid.

\begin{theorem}
Let $C$ be a nonempty closed convex subset of a Hilbert space $H$
and $T$ a nonexpansive mapping of \ $C$ into itself such that
$F(T)\neq \O$. Suppose $x_0\in C$ chosen arbitrarily and $\{x_n\}$
is given by
$$
\left\{\begin{array}{l}y_n=(1-\alpha_n) x_n+\alpha_n Tx_n\\
C_n=\{z\in C:\|y_n-z\|\leq\|x_n-z\|\}\\
Q_n=\{z\in C: \langle z-x_n,x_n-x_0\rangle\geq 0\}\\
x_{n+1}=P_{C_n^{**}\cap Q_n}x_0
\end{array}\right.
$$
where $\{\alpha_n\}$ is chosen such that $0<\alpha\leq\alpha_n\leq
1$. Then, $\{x_n\}$ converges strongly to $P_{F(T)}x_0$.
\end{theorem}

Clearly, Theorem 4.2 is the same as Theorem 1.1.

Moreover, if $C_n^*$ is a closed convex set satisfies $F(T)\subset
C_n^*\subset C_n$, then, $F(T)\subset C_n^*\subset {^*C_n}$.
Therefore, we obtain the following theorem.

\begin{theorem}
Let $C$ be a nonempty closed convex subset of a Hilbert space $H$
and $T$ a nonexpansive mapping of \ $C$ into itself such that
$F(T)\neq \O$. Suppose $x_0\in C$ chosen arbitrarily and $\{x_n\}$
is given by
$$
\left\{\begin{array}{l}y_n=(1-\alpha_n) x_n+\alpha_n Tx_n\\
C_n=\{z\in C:\|y_n-z\|\leq\|x_n-z\|\}\\
Q_n=\{z\in C: \langle z-x_n,x_n-x_0\rangle\geq 0\}\\
x_{n+1}=P_{C_n^*\cap Q_n}x_0
\end{array}\right.
$$
where $C_n^*$ is a closed convex set with $F(T)\subset C_n^*\subset
C_n$ and $\{\alpha_n\}$ is chosen such that
$0<\alpha\leq\alpha_n\leq 1$. Then, $\{x_n\}$ converges strongly to
$P_{F(T)}x_0$.
\end{theorem}

Theorem 4.3 is one of the generalized CQ algorithms in this paper.
Likewise, we can yield many other similar algorithms for Mann,
Ishikawa and Halpern iterations, respectively. They will be proposed
in the following three sections. Some of them are the extensions of
previous results. However, others are obtained directly based on the
framework.

\section{Generalized CQ algorithms for Mann's iteration
process}

In this section, we proposed some algorithms for Mann's iteration
process. To prove the main theorems, we need the following lemmas.

\begin{lemma}
Let $C$ be a nonempty closed convex subset of a real Hilbert space
$H$. Let $T:C\rightarrow C$ be a Lipschitz pseudo-contractive
mapping with Lipschitz constant $L\geq 1$. $\forall x\in C$,
$\alpha\in (0,\frac{1}{L+1})$ and $\tau\in (0,1]$, let
$$
y=(1-\alpha)x+\alpha Tx,
$$
$$
C_x=\{z\in C: \tau\alpha[1-(1+L)\alpha]\|x-T x\|^2\leq\langle
x-z,y-Ty\rangle\}
$$
and
$$
^*C_x=\{z\in C: \|x-T
x\|\leq\frac{(L+1)\alpha+1}{\tau\alpha[1-(L+1)\alpha]}\|x-z\|\}.
$$
Then, there holds $C_x$ is a closed convex set with $F(T)\subset
C_x\subset {^*C_x}$.
\end{lemma}

\begin{proof}
Obviously, $C_x$ is closed and convex. From \cite{8}, we have
\begin{equation}
\alpha[1-(L+1)\alpha]\|x-Tx\|^2\leq\langle x-p,y-Ty\rangle, \
\forall p\in F(T).
\end{equation}
Since $\tau\in (0,1]$, we obtain
\begin{equation}
\tau\alpha[1-(L+1)\alpha]\|x-Tx\|^2\leq\langle x-p,y-Ty\rangle.
\end{equation}
From (5.2), we can conclude that $F(T)\subset C_x$. Let $u\in C_x$,
we have $\forall x\in C$
\begin{equation}
\aligned \tau\alpha[1-(L+1)\alpha]\|x-Tx\|^2&\leq\langle
x-u,y-Ty\rangle\\
&\leq\|x-u\|\|y-Ty\|\\
&\leq\|x-u\|[\|y-x\|+\|x-Tx\|+\|Tx-Ty\|]\\
&\leq\|x-u\|[(L+1)\|x-y\|+\|x-Tx\|]\\
&=[(L+1)\alpha+1]\|x-u\|\|x-Tx\|.
\endaligned
\end{equation}
From the assumption of the coefficients we have
\begin{equation}
\|x-Tx\|\leq\frac{(L+1)\alpha+1}{\tau\alpha[1-(L+1)\alpha]}\|x-u\|
\end{equation}
which implies $u\in {^*C_x}$. So, $F(T)\subset C_x\subset {^*C_x}$.
\end{proof}

\begin{lemma}
Let $C$ be a nonempty closed convex subset of a real Hilbert space
$H$. Let $T:C\rightarrow C$ be a Lipschitz pseudo-contractive
mapping with Lipschitz constant $L\geq 1$ and $F(T)\neq\O$. $\forall
x\in C$, $\alpha\in (0,\frac{1}{L+1})$ and $\tau\in (0,1]$, let
\begin{equation}
y=(1-\alpha)x+\alpha Tx,
\end{equation}
$$
C_x=\{z\in C: \tau\|\alpha (I-T)y\|^2\leq 2\alpha\langle
x-z,(I-T)y\rangle\}.
$$
and
$$
^*C_x=\{z\in C:
\|x-Tx\|\leq\frac{2}{\tau\alpha[1-(L+1)\alpha]}\|x-z\|\}.
$$
Then, there holds $C_x$ is a closed convex set with $F(T)\subset
C_x\subset {^*C_x}$.
\end{lemma}

\begin{proof}
Obviously, $C_x$ is closed and convex. From \cite{10}, we have
$$
\|\alpha(I-T)y\|^2\leq 2\alpha\langle x-p,(I-T)y\rangle, \ \forall
p\in F(T).
$$
since $\tau\in (0,1]$, we obtain
$$
\tau\|\alpha(I-T)y\|^2\leq 2\alpha\langle x-p,(I-T)y\rangle.
$$
which implies that $p\in C_x$, i.e., $F(T)\subset C_x$. Let $u\in
C_x$, then $\forall x\in C$
\begin{equation}
\aligned \tau\|\alpha(I-T)y\|^2&\leq 2\alpha\langle x-u,(I-T)y\rangle\\
&\leq 2\alpha\|x-u\|\|(I-T)y\|.
\endaligned
\end{equation}
It follows that
\begin{equation}
\|y-Ty\|\leq \frac{2}{\tau\alpha}\|x-u\|.
\end{equation}
On the other hand, we have
\begin{equation}
\aligned \|x-Tx\|&\leq\|x-y\|+\|y-Ty\|+\|Ty-Tx\|\\
&\leq(L+1)\alpha\|x-Tx\|+\|y-Ty\|.
\endaligned
\end{equation}
Substitute (5.7) into (5.8), together with the assumption of
coefficients, we get
\begin{equation}
\|x-Tx\|\leq\frac{2}{\tau\alpha[1-(L+1)\alpha]}\|x-u\|
\end{equation}
which implies $u\in {^*C_x}$. So, $F(T)\subset C_x\subset {^*C_x}$.
\end{proof}

\begin{lemma}
Let $C$ be a nonempty closed convex subset of a real Hilbert space
$H$. Let $T$ be a $\kappa-$strict pseudo-contraction of $C$ into
itself for some $0\leq\kappa<1$ with $F(T)\neq\O$. $\forall x\in C$
and $\alpha\in (0,1]$, let
$$
y=(1-\alpha)x+\alpha Tx,
$$
$$
C_x=\{z\in
C:\|y-z\|^2\leq\|x-z\|^2+\alpha(\kappa-(1-\alpha))\|x-Tx\|^2\}
$$
and
$$
^*C_x=\{z\in C:\|x-Tx\|\leq\frac{2}{1-\kappa}\|x-z\|\}.
$$
Then, $C_x$ is a closed convex subset of $C$ with $F(T)\subset
C_x\subset {^*C_x}$.
\end{lemma}

\begin{proof}
By Lemma 2.3, $C_x$ is closed and convex. Let $p\in F(T)$, for any
$x\in C$ we have
\begin{equation}
\aligned \|y-p\|^2&=\|(1-\alpha)(x-p)+\alpha(Tx-p)\|^2\\
&=(1-\alpha)\|x-p\|^2+\alpha\|Tx-p\|^2-\alpha(1-\alpha)\|x-Tx\|^2\\
&\leq(1-\alpha)\|x-p\|^2+\alpha(\|x-p\|^2+\kappa\|x-Tx\|^2)\\
&\quad-\alpha(1-\alpha)\|x-Tx\|^2\\
&=\|x-p\|^2+\alpha(\kappa-(1-\alpha))\|x-Tx\|^2.
\endaligned
\end{equation}
Hence, $F(T)\subset C_x$. Let $u\in C_x$, then $\forall x\in C$, we
obtain
\begin{equation}
\|y-u\|^2\leq\|x-u\|^2+\alpha(\kappa-(1-\alpha))\|x-Tx\|^2
\end{equation}
Besides, we have
\begin{equation}
\|y-u\|^2=(1-\alpha)\|x-u\|^2+\alpha\|Tx-u\|^2-\alpha(1-\alpha)\|x-Tx\|^2.
\end{equation}
Substitute (5.11) into (5.12) to get
\begin{equation}
\alpha\|Tx-u\|^2\leq\alpha\|x-u\|^2+\alpha\kappa\|x-Tx\|^2.
\end{equation}
Since $\alpha>0$, we have
\begin{equation}
\|Tx-u\|^2\leq\|x-u\|^2+\kappa\|x-Tx\|^2.
\end{equation}
On the other hand, we compute
\begin{equation}
\|Tx-u\|^2=\|Tx-x\|^2+2\langle Tx-x,x-u\rangle+\|x-u\|^2.
\end{equation}
Combining (5.14) and (5.15) yields
\begin{equation}
(1-\kappa)\|x-Tx\|^2\leq 2\langle x-Tx,x-u\rangle\leq
2\|x-Tx\|\|x-u\|.
\end{equation}
Since $\kappa<1$, then
\begin{equation}
\|x-Tx\|\leq\frac{2}{1-\kappa}\|x-u\|
\end{equation}
which implies $u\in {^*C_x}$. So, $F(T)\subset C_x\subset {^*C_x}$.
\end{proof}

Using Lemma 5.1, we obtain the following theorem.

\begin{theorem}
Let $C$ be a nonempty closed convex subset of a real Hilbert space
$H$ and $T$ a Lipschitz pseudo-contraction from $C$ into itself with
the Lipschitz constant $L\geq 1$ and $F(T)\neq\O$. Assume sequence
$\{\tau_n\}\subset [\tau,1]$ with $\tau\in (0,1]$ and sequence
$\{\alpha_n\}\subset [a,b]$ with $a,b\in (0,\frac{1}{L+1})$. Let
$\{x_n\}$ be a sequence generated by the following manner:
$$
\left\{\begin{array}{l}x_0\in C \ chosen \ arbitrarily\\
y_n=(1-\alpha_n) x_n+\alpha_n Tx_n\\
C_n=\{z\in C: \tau_n\alpha_n[1-(1+L)\alpha_n]\|x_n-T x_n\|^2\leq\langle x_n-z,y_n-Ty_n\rangle\}\\
Q_n=\{z\in C: \langle z-x_n,x_n-x_0\rangle\geq 0\}\\
x_{n+1}=P_{C_n^*\cap Q_n}x_0
\end{array}\right.
$$
where $C_n^*$ is a closed convex set with $F(T)\subset C_n^*\subset
C_n$. Then $\{x_n\}$ converges strongly to $P_{F(T)}x_0$.
\end{theorem}

\begin{proof}
Let $^*C_n=\{z\in C: \|x_n-T
x_n\|\leq\frac{(L+1)\alpha_n+1}{\tau_n\alpha_n[1-(L+1)\alpha_n]}\|x_n-z\|\}$,
then using Lemma 5.1, we obtain $F(T)\subset C_n^*\subset C_n\subset
{^*C_n}$. From the assumption,
$\frac{(L+1)\alpha_n+1}{\tau_n\alpha_n[1-(L+1)\alpha_n]}\leq\frac{(L+1)b+1}{\tau
a[1-(L+1)b]}<\infty$. By Lemma 2.5 and Theorem 3.1, we can prove
$x_n\rightarrow P_{F(T)}x_0$.
\end{proof}

We can prove the following theorem based on Lemma 5.2.

\begin{theorem}
Let $C$ be a nonempty closed convex subset of a real Hilbert space
$H$. Let $T:C\rightarrow C$ be a L-Lipschitz pseudo-contractive
mapping such that $L\geq 1$ and $F(T)\neq\O$. Assume sequence
$\{\tau_n\}\subset [\tau,1]$ with $\tau\in (0,1]$ and sequence
$\{\alpha_n\}\subset [a,b]$ with $a,b\in (0,\frac{1}{L+1})$. Suppose
$x_0\in C$ chosen arbitrarily and $\{x_n\}$ is given by
$$
\left\{\begin{array}{l}y_n=(1-\alpha_n)x_n+\alpha_n Tx_n\\
C_n=\{z\in C: \tau_n\|\alpha_n(I-T)y_n\|^2\leq 2\alpha_n\langle x_n-z,(I-T)y_n\rangle\}\\
Q_n=\{z\in C: \langle z-x_n,x_n-x_0\rangle\geq 0\}\\
x_{n+1}=P_{C_n^*\cap Q_n}x_0
\end{array}\right.
$$
where $C_n^*$ is a closed convex set with $F(T)\subset C_n^*\subset
C_n$. Then, $\{x_n\}$ converges strongly to $P_{F(T)}x_0$.
\end{theorem}

\begin{proof}
From the assumption, we have
$\frac{2}{\tau_n\alpha_n[1-(L+1)\alpha_n]}\leq\frac{2}{\tau
a[1-(L+1)b]}<\infty$. Let $^*C_n=\{z\in C:
\|x_n-Tx_n\|\leq\frac{2}{\tau_n\alpha_n[1-(L+1)\alpha_n]}\|x_n-z\|\}$,
using Lemma 5.2, we can conclude $F(T)\subset C_n^*\subset
C_n\subset {^*C_n}$. Hence, by Lemma 2.5 and Theorem 3.1,
$x_n\rightarrow P_{F(T)}x_0$.
\end{proof}

\begin{remark}
In fact, it is easily to prove Theorem 5.2 by Theorem 5.1 directly.
\end{remark}

By Lemma 5.3, the following theorem is valid.

\begin{theorem}
Let $C$ be a nonempty closed convex subset of a Hilbert space $H$
and $T$ a $\kappa-$strict pseudo-contraction of $C$ into itself for
some $0\leq\kappa<1$ with $F(T)\neq\O$. Suppose $x_0\in C$ chosen
arbitrarily and $\{x_n\}$ is given by
$$
\left\{\begin{array}{l}y_n=(1-\alpha_n) x_n+\alpha_n Tx_n\\
C_n=\{z\in C: \|y_n-z\|^2\leq\|x_n-z\|^2+\alpha_n(\kappa-(1-\alpha_n))\|x_n-Tx_n\|^2\}\\
Q_n=\{z\in C: \langle z-x_n,x_n-x_0\rangle\geq 0\}\\
x_{n+1}=P_{C_n^*\cap Q_n}x_0
\end{array}\right.
$$
where $C_n^*$ is a closed convex set with $F(T)\subset C_n^*\subset
C_n$ and $\{\alpha_n\}$ is chosen such that
$0<\alpha\leq\alpha_n\leq 1$. Then, $\{x_n\}$ converges strongly to
$P_{F(T)}x_0$.
\end{theorem}

\begin{proof}
Clearly, $\frac{2}{1-\kappa}<\infty$. Let ${^*C_n}=\{z\in
C:\|x_n-Tx_n\|\leq\frac{2}{1-\kappa}\|x-z\|\}$, then using Lemma
5.3, $F(T)\subset C_n^*\subset C_n\subset {^*C_n}$. Hence, using
Lemma 2.6 and Theorem 3.1, $x_n\rightarrow P_{F(T)}x_0$.
\end{proof}

\begin{corollary}
Let $C$ be a nonempty closed convex subset of a Hilbert space $H$
and $T$ a nonexpansive mapping of $C$ into itself with $F(T)\neq\O$.
Suppose $x_0\in C$ chosen arbitrarily and $\{x_n\}$ is given by
$$
\left\{\begin{array}{l}y_n=(1-\alpha_n) x_n+\alpha_n Tx_n\\
C_n=\{z\in C: \|y_n-z\|^2\leq\|x_n-z\|^2-\alpha_n(1-\alpha_n)\|x_n-Tx_n\|^2\}\\
Q_n=\{z\in C: \langle z-x_n,x_n-x_0\rangle\geq 0\}\\
x_{n+1}=P_{C_n^*\cap Q_n}x_0
\end{array}\right.
$$
where $C_n^*$ is a closed convex set with $F(T)\subset C_n^*\subset
C_n$ and $\{\alpha_n\}$ is chosen such that
$0<\alpha\leq\alpha_n\leq 1$. Then, $\{x_n\}$ converges strongly to
$P_{F(T)}x_0$.
\end{corollary}

\begin{remark}
Corollary 5.4 is a deduced result of Theorem 5.3. In these two
theorems, set $C_n^*=C_n$, then we obtain two algorithms which were
also proposed in \cite{5}. Corollary 5.4 is also the deduced result
of Theorem 4.3.
\end{remark}

\section{Generalized CQ algorithms for Ishikawa's iteration
process}

In this section, we introduce some algorithms for Ishikawa's
iteration process. To prove the main theorems, we need the following
lemmas.

\begin{lemma}
Let $C$ be a nonempty closed convex subset of a real Hilbert space
$H$. Let $T:C\rightarrow C$ be a Lipschitz pseudo-contractive
mapping with Lipschitz constant $L\geq 1$ and $F(T)\neq\O$. $\forall
x\in C$ and $\alpha,\beta\in (0,1)$ such that
$0<\beta\leq\alpha<\frac{1}{\sqrt{1+L^2}+1}$, let
$$
\aligned
&v=(1-\alpha)x+\alpha Tx\\
&y=(1-\beta)x+\beta Tv,
\endaligned
$$
$$
C_x=\{z\in
C:\|y-z\|^2\leq\|x-z\|^2-\alpha\beta(1-2\alpha-L^2\alpha^2)\|x-Tx\|^2\}
$$
and
$$
^*C_x=\{z\in
C:\|x-Tx\|\leq\frac{2(1+L\alpha)}{\alpha(1-2\alpha-L^2\alpha^2)}\|x-z\|\}.
$$
Then, there holds $C_x$ is a closed convex set with $F(T)\subset
C_x\subset {^*C_x}$.
\end{lemma}

\begin{proof}
Obviously, by Lemma 2.3, we can conclude $C_x$ is closed and convex.
From \cite{11}, we can easily obtain $F(T)\subset C_x$. Taking $u\in
C_x$, $\forall x\in C$, we get
\begin{equation}
\|y-u\|^2\leq\|x-u\|^2-\alpha\beta(1-2\alpha-L^2\alpha^2)\|x-Tx\|^2.
\end{equation}
On the other hand,
\begin{equation}
\|y-u\|^2=\|y-x\|^2+2\langle y-x,x-u\rangle+\|x-u\|^2.
\end{equation}
Combining (6.1) and (6.2), we have
\begin{equation}
\alpha(1-2\alpha-L^2\alpha^2)\|x-Tx\|^2\leq 2\langle x-Tv,x-u\rangle
\end{equation}
It follows that,
\begin{equation}
\aligned \alpha(1-2\alpha-L^2\alpha^2)\|x-Tx\|^2&\leq 2\langle
x-Tv,x-u\rangle\\
&\leq2\|x-u\|\|x-Tv\|\\
&\leq 2\|x-u\|(\|x-Tx\|+\|Tx-Tv\|)\\
&\leq 2(1+L\alpha)\|x-u\|\|x-Tx\|.
\endaligned
\end{equation}
Noting that the function $f(x)=1-2t-L^2 t^2$ is strictly decreasing
in $t\in (0,1)$, we infer that
$$
1-2\alpha-L^2\alpha^2>0.
$$
Then, from (6.4), we have
\begin{equation}
\|x-Tx\|\leq\frac{2(1+L\alpha)}{\alpha(1-2\alpha-L^2\alpha^2)}\|x-u\|
\end{equation}
which implies $u\in {^*C_x}$. So, $F(T)\subset C_x\subset {^*C_x}$.
\end{proof}

\begin{lemma}
Let $C$ be a nonempty closed convex subset of a real Hilbert space
$H$. Let $T:C\rightarrow C$ be a $\kappa-$strict pseudo-contractive
mapping for some $0\leq\kappa<1$ with $F(T)\neq\O$. $\forall x\in C$
and $\alpha,\beta\in [0,1]$ such that
$0<\alpha<\frac{2}{\sqrt{4\kappa L^2+(\kappa+1)^2}+(\kappa+1)}$ and
$0<\beta\leq\kappa\alpha+(1-\kappa)$, let
$$
\aligned
&v=(1-\alpha)x+\alpha Tx\\
&y=(1-\beta)x+\beta Tv,
\endaligned
$$
$$
C_x=\{z\in
C:\|y-z\|^2\leq\|x-z\|^2-\alpha\beta[1-(\kappa+1)\alpha-\kappa
L^2\alpha^2]\|x-Tx\|^2\}
$$
and
$$
^*C_x=\{z\in
C:\|x-Tx\|\leq\frac{2(1+L\alpha)}{\alpha[1-(\kappa+1)\alpha-\kappa
L^2\alpha^2]}\|x-z\|\}.
$$
Then, there holds $C_x$ is a closed convex set with $F(T)\subset
C_x\subset {^*C_x}$.
\end{lemma}

\begin{proof}
Obviously, by Lemma 2.3, we can conclude $C_x$ is closed and convex.
Let $p\in F(T)$. We have,
\begin{equation}
\aligned \|v-p\|^2&=\|(1-\alpha)(x-p)+\alpha(Tx-p)\|^2\\
&=(1-\alpha)\|x-p\|^2+\alpha\|Tx-p\|^2-\alpha(1-\alpha)\|x-Tx\|^2\\
&\leq(1-\alpha)\|x-p\|^2+\alpha(\|x-p\|^2+\kappa\|x-Tx\|^2)\\
&\quad -\alpha(1-\alpha)\|x-Tx\|^2\\
&=\|x-p\|^2+\alpha[\kappa-(1-\alpha)]\|x-Tx\|^2
\endaligned
\end{equation}
and
\begin{equation}
\aligned
\|v-Tv\|^2&=\|(1-\alpha)(x-Tv)+\alpha(Tx-Tv)\|^2\\
&=(1-\alpha)\|x-Tv\|^2+\alpha\|Tx-Tv\|^2-\alpha(1-\alpha)\|x-Tx\|^2\\
&\leq(1-\alpha)\|x-Tv\|^2+L^2\alpha\|x-v\|^2-\alpha(1-\alpha)\|x-Tx\|^2\\
&=(1-\alpha)\|x-Tv\|^2+L^2\alpha^3\|x-Tx\|^2-\alpha(1-\alpha)\|x-Tx\|^2\\
&=(1-\alpha)\|x-Tv\|^2+\alpha(L^2\alpha^2+\alpha-1)\|x-Tx\|^2
\endaligned
\end{equation}
and also,
\begin{equation}
\aligned \|y-p\|^2&=\|(1-\beta)(x-p)+\beta(Tv-p)\|^2\\
&=(1-\beta)\|x-p\|^2+\beta\|Tv-p\|^2-\beta(1-\beta)\|x-Tv\|^2\\
&\leq(1-\beta)\|x-p\|^2+\beta(\|v-p\|^2+\kappa\|v-Tv\|^2)-\beta(1-\beta)\|x-Tv\|^2
\endaligned
\end{equation}
Substituting (6.6) and (6.7) in (6.8), we yield
$$
\aligned \|y-p\|^2&\leq(1-\beta)\|x-p\|^2+\beta\|x-p\|^2+\beta\alpha[\kappa-(1-\alpha)]\|x-Tx\|^2\\
&\quad+\beta\kappa(1-\alpha)\|x-Tv\|^2+\beta\kappa\alpha(L^2\alpha^2+\alpha-1)\|x-Tx\|^2\\
&\quad-\beta(1-\beta)\|x-Tv\|^2\\
&=\|x-p\|^2+\beta[\kappa(1-\alpha)-(1-\beta)]\|x-Tv\|^2\\
&\quad+\alpha\beta[\kappa L^2\alpha^2+(\kappa+1)\alpha-1]\|x-Tx\|^2.
\endaligned
$$
Since $\beta\leq\kappa\alpha+(1-\kappa)$, then,
$$
\aligned
\|y-p\|^2&\leq\|x-p\|^2+\alpha\beta[\kappa L^2\alpha^2+(\kappa+1)\alpha-1]\|x-Tx\|^2\\
&=\|x-p\|^2-\alpha\beta[1-(\kappa+1)\alpha-\kappa
L^2\alpha^2]\|x-Tx\|^2.
\endaligned
$$
Therefore, $p\in C_x$, i.e., $F(T)\subset C_x$. Taking $u\in C_x$,
$\forall x\in C$, we get
\begin{equation}
\|y-u\|^2\leq\|x-u\|^2-\alpha\beta[1-(\kappa+1)\alpha-\kappa
L^2\alpha^2]\|x-Tx\|^2.
\end{equation}
On the other hand,
\begin{equation}
\|y-u\|^2=\|y-x\|^2+2\langle y-x,x-u\rangle+\|x-u\|^2.
\end{equation}
Combining (6.9) and (6.10), we have
\begin{equation}
\alpha[1-(\kappa+1)\alpha-\kappa L^2\alpha^2]\|x-Tx\|^2\leq 2\langle
x-Tv,x-u\rangle
\end{equation}
It follows that,
\begin{equation}
\aligned \alpha[1-(\kappa+1)\alpha-\kappa
L^2\alpha^2]\|x-Tx\|^2&\leq 2\langle
x-Tv,x-u\rangle\\
&\leq2\|x-u\|\|x-Tv\|\\
&\leq 2\|x-u\|(\|x-Tx\|+\|Tx-Tv\|)\\
&\leq 2(1+L\alpha)\|x-u\|\|x-Tx\|.
\endaligned
\end{equation}
Noting that the function $f(x)=1-(\kappa+1)t-\kappa L^2 t^2$ is
strictly decreasing in $t\in (0,1)$, we infer that
$$
1-(\kappa+1)\alpha-\kappa L^2\alpha^2>0.
$$
Then, from (6.12), we have
\begin{equation}
\|x-Tx\|\leq\frac{2(1+L\alpha)}{\alpha[1-(\kappa+1)\alpha-\kappa
L^2\alpha^2]}\|x-u\|
\end{equation}
which implies $u\in {^*C_x}$. So, $F(T)\subset C_x\subset {^*C_x}$.
\end{proof}

\begin{lemma}
Let $C$ be a nonempty closed convex subset of a real Hilbert space
$H$. Let $T$ be a nonexpansive mapping of \ $C$ into itself with
$F(T)\neq\O$. $\forall x\in C$, $0<\beta\leq 1$ and $0\leq\alpha\leq
1$ let
$$
\aligned
&v=(1-\alpha)x+\alpha Tx\\
&y=(1-\beta)x+\beta Tv,
\endaligned
$$
$$
C_x=\{z\in C: \|z-y\|^2\leq\|z-x\|^2+\beta(\|v\|^2-\|x\|^2+2\langle
x-v,z\rangle)\}
$$
and
$$
^*C_x=\{z\in C: \beta(1-\alpha)\|x-Tx\|\leq
3\|x-z\|+\alpha\|Tx-z\|\}.
$$
Then, $C_x$ is a closed convex subset with $F(T)\subset C_x\subset
{^*C_x}$.
\end{lemma}

\begin{proof}
By Lemma 2.3, we see that $C_x$ is closed and convex. Let $p\in
F(T)$, for any $x\in C$ we have,
$$
\aligned\|y-p\|^2&=\|(1-\beta)(x-p)+\beta(Tv-p)\|^2\\
&\leq(1-\beta)\|x-p\|^2+\beta\|v-p\|^2\\
&=\|x-p\|^2+\beta(\|v-p\|^2-\|x-p\|^2)\\
&=\|x-p\|^2+\beta(\|v\|^2-\|x\|^2+2\langle x-v,p\rangle)
\endaligned
$$
Hence, $F(T)\subset C_x$. Let $u\in C_x$, then we get
\begin{equation}
\aligned \|y-u\|^2&\leq\|x-u\|^2+\beta(\|v\|^2-\|x\|^2+2\langle x-v,u\rangle)\\
&=\|x-u\|^2+\beta(\|v-x\|^2+2\langle x-v,u-x\rangle)\\
&\leq\|x-u\|^2+\beta\|v-u\|^2\\
&\leq[\|x-u\|+\sqrt{\beta}\|v-u\|]^2.
\endaligned
\end{equation}
It follows that,
\begin{equation}
\aligned\|y-u\|&\leq\|x-u\|+\sqrt{\beta}\|v-u\|\\
&=\|x-u\|+\sqrt{\beta}\|(1-\alpha)(x-u)+\alpha(Tx-u)\|\\
&\leq\|x-u\|+\sqrt{\beta}(1-\alpha)\|x-u\|+\sqrt{\beta}\alpha\|Tx-u\|\\
&\leq 2\|x-u\|+\alpha\|Tx-u\|.
\endaligned
\end{equation}
Besides,
\begin{equation}
\aligned \|x-Tx\|&\leq\|x-Tv\|+\|Tv-Tx\|\\
&\leq\frac{1}{\beta}\|x-y\|+\|v-x\|\\
&\leq\frac{1}{\beta}[\|x-u\|+\|y-u\|]+\alpha\|x-Tx\|.
\endaligned
\end{equation}
Combining (6.15) and (6.16), we obtain
\begin{equation}
\aligned \beta(1-\alpha)\|x-Tx\|&\leq\|x-u\|+\|y-u\|\\
&\leq 3\|x-u\|+\alpha\|Tx-u\|.
\endaligned
\end{equation}
From (6.17), we can conclude $u\in {^*C_x}$. So, $F(T)\subset
C_x\subset {^*C_x}$.
\end{proof}

\begin{lemma}
Let $C$ be a nonempty closed convex subset of a real Hilbert space
$H$. Let $T$ be a nonexpansive mapping of \ $C$ into itself with
$F(T)\neq\O$. $\forall x\in C$, $0\leq\alpha<1$ and
$0\leq\frac{\alpha}{1-\alpha}<\beta\leq 1$, let
$$
\aligned
&v=(1-\alpha)x+\alpha Tx\\
&y=(1-\beta)x+\beta Tv,
\endaligned
$$
$$
C_x=\{z\in C: \|z-y\|^2\leq\|z-x\|^2+\beta(\|v\|^2-\|x\|^2+2\langle
x-v,z\rangle)\}
$$
and
$$
^*C_x=\{z\in C: \|x-Tx\|\leq
\frac{4}{\beta(1-\alpha)-\alpha}\|x-z\|\}.
$$
Then, $C_x$ is a closed convex subset with $F(T)\subset C_x\subset
{^*C_x}$.
\end{lemma}

\begin{proof}
By Lemma 2.3, we see that $C_x$ is closed and convex and similarly,
we can prove $F(T)\subset C_x$. Let $u\in C_x$, then we get
\begin{equation}
\aligned \|y-u\|^2&\leq\|x-u\|^2+\beta(\|v\|^2-\|x\|^2+2\langle x-v,u\rangle)\\
&\leq[\|x-u\|+\sqrt{\beta}\|v-u\|]^2.
\endaligned
\end{equation}
It follows that,
\begin{equation}
\aligned\|y-u\|&\leq\|x-u\|+\sqrt{\beta}\|v-u\|\\
&=\|x-u\|+\sqrt{\beta}\|(1-\alpha)(x-u)+\alpha(Tx-u)\|\\
&\leq\|x-u\|+\sqrt{\beta}(1-\alpha)\|x-u\|+\sqrt{\beta}\alpha\|Tx-u\|\\
&\leq 2\|x-u\|+\alpha\|Tx-u\|\\
&\leq 2\|x-u\|+\alpha\|Tx-x\|+\alpha\|x-u\|\\
&\leq 3\|x-u\|+\alpha\|x-Tx\|.
\endaligned
\end{equation}
Besides,
\begin{equation}
\aligned \|x-Tx\|&\leq\|x-Tv\|+\|Tv-Tx\|\\
&\leq\frac{1}{\beta}\|x-y\|+\|v-x\|\\
&\leq\frac{1}{\beta}[\|x-u\|+\|y-u\|]+\alpha\|x-Tx\|.
\endaligned
\end{equation}
Combining (6.19) and (6.20), we obtain
\begin{equation}
\aligned \beta(1-\alpha)\|x-Tx\|&\leq\|x-u\|+\|y-u\|\\
&\leq 4\|x-u\|+\alpha\|x-Tx\|.
\endaligned
\end{equation}
Hence,
\begin{equation}
\|x-Tx\|\leq \frac{4}{\beta(1-\alpha)-\alpha}\|x-u\|.
\end{equation}
From (6.22), we can conclude $u\in {^*C_x}$. So, $F(T)\subset
C_x\subset {^*C_x}$
\end{proof}

\begin{theorem}
Let $C$ be a nonempty closed convex subset of a real Hilbert space
$H$. Let $T:C\rightarrow C$ be a Lipschitz pseudo-contraction such
that $L\geq 1$ and $F(T)\neq\O$. Suppose that $\{\alpha_n\}$ and
$\{\beta_n\}$ are two real sequences in $(0,1)$ satisfying the
conditions:
\\
($C_1$) $0<\beta_n\leq\alpha_n$, $\forall n\geq 0$;\\
($C_2$) $\liminf_{n\rightarrow\infty}\alpha_n\geq\alpha'>0$;\\
($C_3$)
$\limsup_{n\rightarrow\infty}\alpha_n\leq\alpha<\frac{1}{\sqrt{1+L^2}+1},
\forall n\geq 0$.
\\
Let a sequence $\{x_n\}$ be generated by
$$
\left\{\begin{array}{l}x_0\in C \ chosen \ arbitrarily\\
v_n=(1-\alpha_n)x_n+\alpha_n Tx_n\\
y_n=(1-\beta_n)x_n+\beta_n Tv_n\\
C_n=\{z\in C:\|y_n-z\|^2\leq\|x_n-z\|^2 -\alpha_n\beta_n
(1-2\alpha_n-L^2 \alpha_n^2)\|x_n-Tx_n\|^2\}\\
Q_n=\{z\in C:\langle z-x_n,x_n-x_0\rangle\geq 0\}\\
x_{n+1}=P_{C_n^*\cap Q_n}\\
\end{array}\right.
$$
where $C_n^*$ is a closed convex set with $F(T)\subset C_n^*\subset
C_n$. Then, $\{x_n\}$ converges strongly to a fixed point
$P_{F(T)}x_0$.
\end{theorem}

\begin{proof}
From the assumption, we have
$\frac{2(1+L\alpha_n)}{\alpha_n(1-2\alpha_n-L^2\alpha_n^2)}\leq\frac{2(1+L\alpha)}{\alpha'(1-2\alpha-L^2\alpha^2)}<\infty$.
Let $^*C_n=\{z\in
C:\|x_n-Tx_n\|\leq\frac{2(1+L\alpha_n)}{\alpha_n(1-2\alpha_n-L^2\alpha_n^2)}\|x_n-z\|\}$,
then using Lemma 6.1, $F(T)\subset C_n^*\subset C_n\subset {^*C_n}$.
Hence, by Lemma 2.5 and Theorem 3.1, $x_n\rightarrow P_{F(T)}x_0$.
\end{proof}

\begin{remark}
In this theorem, let $C_n^*=C_n$, then we yield a theorem which was
also introduced in \cite{11}.
\end{remark}

\begin{theorem}
Let $C$ be a nonempty closed convex subset of a real Hilbert space
$H$ and $T$ a Lipschitz pseudo-contraction from $C$ into itself with
the Lipschitz constant $L\geq 1$. Assume sequence $\{\tau_n\}\subset
[\tau,1]$ with $\tau\in (0,1]$, sequence $\{\alpha_n\}\subset [a,b]$
with $a,b\in (0,\frac{1}{L+1})$ and sequence $\{\beta_n\}\subset
(0,1)$ satisfies that $\beta_n\leq\alpha_n$. Let $\{x_n\}$ be a
sequence generated by the following manner:
\begin{equation}
\left\{\begin{array}{l}x_0\in C \ chosen \ arbitrarily\\
v_n=(1-\alpha_n) x_n+\alpha_n Tx_n\\
y_n=(1-\beta_n)x_n+\beta_n Tv_n\\
C_n'=\{z\in C: \tau_n\alpha_n[1-(1+L)\alpha_n]\|x_n-T x_n\|^2\leq\langle x_n-z,v_n-Tv_n\rangle\}\\
C_n''=\{z\in
C:\|y_n-z\|^2\leq\|x_n-z\|^2-\alpha_n\beta_n(1-2\alpha_n-L^2\alpha_n^2)\|x_n-Tx_n\|^2\}\\
C_n=C_n'\cap C_n''\\
Q_n=\{z\in C: \langle z-x_n,x_n-x_0\rangle\geq 0\}\\
x_{n+1}=P_{C_n^*\cap Q_n}x_0
\end{array}\right.
\end{equation}
where $C_n^*$ is a closed convex set with $F(T)\subset C_n^*\subset
C_n$. Then $\{x_n\}$ converges strongly to $P_{F(T)}x_0$.
\end{theorem}

\begin{proof}
Obviously, $F(T)\subset C_n^*\subset C_n\subset C_n'$. Then, by
Theorem 5.1, we can obtain $x_n\rightarrow P_{F(T)}x_0$.
\end{proof}

\begin{theorem}
Let $C$ be a nonempty closed convex subset of a real Hilbert space
$H$. Let $T:C\rightarrow C$ be a $\kappa-$strict-pseudo-contraction
for some $0\leq\kappa<1$ such that $F(T)\neq\O$. Suppose that
$\{\alpha_n\}$ and $\{\beta_n\}$ are two real sequences in $[0,1]$
satisfying $0<\alpha'\leq\alpha_n\leq\alpha<\frac{2}{\sqrt{4\kappa
L^2+(\kappa+1)^2}+(\kappa+1)}$ and
$0<\beta_n\leq\kappa\alpha_n+(1-\kappa)$. Suppose sequence $\{x_n\}$
be generated by
$$
\left\{\begin{array}{l}x_0\in C \ chosen \ arbitrarily\\
v_n=(1-\alpha_n)x_n+\alpha_n Tx_n\\
y_n=(1-\beta_n)x_n+\beta_n Tv_n\\
C_n=\{z\in C:\|y_n-z\|^2\leq\|x_n-z\|^2 -\alpha_n\beta_n
(1-(\kappa+1)\alpha_n-\kappa L^2 \alpha_n^2)\|x_n-Tx_n\|^2\}\\
Q_n=\{z\in C:\langle z-x_n,x_n-x_0\rangle\geq 0\}\\
x_{n+1}=P_{C_n^*\cap Q_n}\\
\end{array}\right.
$$
where $C_n^*$ is a closed convex set with $F(T)\subset C_n^*\subset
C_n$. Then, $\{x_n\}$ converges strongly to a fixed point
$P_{F(T)}x_0$.
\end{theorem}

\begin{proof}
From the assumption, we have
$\frac{2(1+L\alpha_n)}{\alpha_n[1-(\kappa+1)\alpha_n-\kappa
L^2\alpha_n^2]}\leq\frac{2(1+L\alpha)}{\alpha'[1-(\kappa+1)\alpha-\kappa
L^2\alpha^2]}<\infty$. Let $^*C_n=\{z\in
C:\|x_n-Tx_n\|\leq\frac{2(1+L\alpha_n)}{\alpha_n[1-(\kappa+1)\alpha_n-\kappa
L^2\alpha_n^2]}\|x_n-z\|\}$, then using Lemma 6.2, $F(T)\subset
C_n^*\subset C_n\subset {^*C_n}$. Hence, by Lemma 2.6 and Theorem
3.1, $x_n\rightarrow P_{F(T)}x_0$.
\end{proof}

\begin{theorem}
Let $C$ be a nonempty closed convex subset of a real Hilbert space
$H$. Let $T:C\rightarrow C$ be a $\kappa-$strict-pseudo-contraction
for some $0\leq\kappa<1$ such that $F(T)\neq\O$. Suppose that
$\{\alpha_n\}$ and $\{\beta_n\}$ are two real sequences in $[0,1]$
satisfying $0<\alpha\leq\alpha_n\leq 1$ and
$0\leq\beta_n\leq\kappa\alpha_n+(1-\kappa)$. Let a sequence
$\{x_n\}$ be generated by
$$
\left\{\begin{array}{l}x_0\in C \ chosen \ arbitrarily\\
v_n=(1-\alpha_n)x_n+\alpha_n Tx_n\\
y_n=(1-\beta_n)x_n+\beta_n Tv_n\\
C_n'=\{z\in C: \|v_n-z\|^2\leq\|x_n-z\|^2+\alpha_n(\kappa-(1-\alpha_n))\|x_n-Tx_n\|^2\}\\
C_n''=\{z\in C:\|y_n-z\|^2\leq\|x_n-z\|^2 -\alpha_n\beta_n
(1-(\kappa+1)\alpha_n-\kappa L^2 \alpha_n^2)\|x_n-Tx_n\|^2\}\\
C_n=C_n'\cap C_n''\\
Q_n=\{z\in C:\langle z-x_n,x_n-x_0\rangle\geq 0\}\\
x_{n+1}=P_{C_n^*\cap Q_n}\\
\end{array}\right.
$$
where $C_n^*$ is a closed convex set with $F(T)\subset C_n^*\subset
C_n$ and $\{\alpha_n\}$ is chosen such that
$0<\alpha\leq\alpha_n\leq 1$. Then, $\{x_n\}$ converges strongly to
$P_{F(T)}x_0$.
\end{theorem}

\begin{proof}
It is clear that $F(T)\subset C_n^*\subset C_n\subset C_n'$. Hence,
by Theorem 5.3, $x_n\rightarrow P_{F(T)}x_0$.
\end{proof}

Using our method, we can yield at least four different CQ algorithms
for Ishikawa's iteration process for nonexpansive mappings.

\begin{theorem}
Let $C$ be a nonempty closed convex subset of a real Hilbert space
$H$. Let $T:C\rightarrow C$ be a nonexpansive mapping with
$F(T)\neq\O$. Suppose that $\{\alpha_n\}$ and $\{\beta_n\}$ are two
real sequences in $[0,1]$ satisfying
$0<\alpha'\leq\alpha_n\leq\alpha<1$ and $0<\beta_n\leq 1$. Let a
sequence $\{x_n\}$ be generated by
$$
\left\{\begin{array}{l}x_0\in C \ chosen \ arbitrarily\\
v_n=(1-\alpha_n)x_n+\alpha_n Tx_n\\
y_n=(1-\beta_n)x_n+\beta_n Tv_n\\
C_n=\{z\in C:\|z-y_n\|^2\leq\|z-x_n\|^2 -\alpha_n\beta_n
(1-\alpha_n)\|x_n-Tx_n\|^2\}\\
Q_n=\{z\in C:\langle z-x_n,x_n-x_0\rangle\geq 0\}\\
x_{n+1}=P_{C_n^*\cap Q_n}\\
\end{array}\right.
$$
where $C_n^*$ is a closed convex set with $F(T)\subset C_n^*\subset
C_n$. Then, $\{x_n\}$ converges strongly to a fixed point
$P_{F(T)}x_0$.
\end{theorem}

\begin{proof}
By Theorem 6.3, we can prove the conclusion.
\end{proof}

\begin{theorem}
Let $C$ be a nonempty closed convex subset of a Hilbert space $H$
and $T$ a nonexpansive mapping of \ $C$ into itself such that
$F(T)\neq \O$. Assume that $\{\alpha_n\}$ and $\{\beta_n\}$ are
sequences in $[0,1]$ such that $0<\beta\leq\beta_n\leq 1$ and
$\alpha_n\rightarrow 0$. Define a sequence $\{x_n\}$ in $C$ by
algorithm:
$$
\left\{\begin{array}{l}x_0\in C \ chosen \ arbitrarily\\
v_n=(1-\alpha_n)x_n+\alpha_n Tx_n\\
y_n=(1-\beta_n)x_n+\beta_n Tv_n\\
C_n=\{z\in C: \|z-y_n\|^2\leq\|z-x_n\|^2+\beta_n(\|v_n\|^2-\|x_n\|^2+2\langle x_n-v_n,z\rangle)\}\\
Q_n=\{z\in C: \langle z-x_n,x_n-x_0\rangle\geq 0\}\\
x_{n+1}=P_{C_n^*\cap Q_n}x_0
\end{array}\right.
$$
where $C_n^*$ is a closed convex set with $F(T)\subset C_n^*\subset
C_n$. Then, $\{x_n\}$ converges strongly to $P_{F(T)}x_0$.
\end{theorem}

\begin{proof}
First observing that $\alpha_n\rightarrow 0$, we can conclude
$\alpha_n\neq 1$ since $n$ is sufficient big. So, without losing
generality, we can assume $0\leq\alpha_n\leq\alpha<1$. Combining
with the assumption of $\beta_n$, we have $\frac{3}{\beta_n
(1-\alpha_n)}\leq\frac{3}{\beta(1-\alpha)}<\infty$. Easily, we can
prove $\|Tx_n-x_{n+1}\|$ is bounded. Then,
$\lim_{n\rightarrow\infty}\frac{\alpha_n}{\beta_n
(1-\alpha_n)}\|Tx_n-x_{n+1}\|=0$. Let $^*C_n=\{z\in C:
\beta_n(1-\alpha_n)\|x_n-Tx_n\|\leq
3\|x_n-z\|+\alpha_n\|Tx_n-z\|\}$, then using Lemma 6.3, $F(T)\subset
C_n^*\subset C_n\subset {^*C_n}$. Hence, by Lemma 2.6 and Theorem
3.1, $x_n\rightarrow P_{F(T)}x_0$.
\end{proof}

\begin{remark}
In this theorem, let $C_n^*=C_n$, then we obtain an algorithm which
was also proposed in \cite{9}.
\end{remark}

\begin{theorem}
Let $C$ be a nonempty closed convex subset of a Hilbert space $H$
and $T$ a nonexpansive mapping of \ $C$ into itself such that
$F(T)\neq \O$. Assume that $\{\alpha_n\}$ and $\{\beta_n\}$ are
sequences in $[0,1]$ such that $0\leq\alpha_n\leq\alpha<1$ and
$0\leq\frac{\alpha_n}{1-\alpha_n}\leq\frac{\alpha}{1-\alpha}<\beta\leq\beta_n\leq
1$. Define a sequence $\{x_n\}$ in $C$ by algorithm:
$$
\left\{\begin{array}{l}x_0\in C \ chosen \ arbitrarily\\
v_n=(1-\alpha_n)x_n+\alpha_n Tx_n\\
y_n=(1-\beta_n)x_n+\beta_n Tv_n\\
C_n=\{z\in C: \|z-y_n\|^2\leq\|z-x_n\|^2+\beta_n(\|v_n\|^2-\|x_n\|^2+2\langle x_n-v_n,z\rangle)\}\\
Q_n=\{z\in C: \langle z-x_n,x_n-x_0\rangle\geq 0\}\\
x_{n+1}=P_{C_n^*\cap Q_n}x_0
\end{array}\right.
$$
where $C_n^*$ is a closed convex set with $F(T)\subset C_n^*\subset
C_n$. Then, $\{x_n\}$ converges strongly to $P_{F(T)}x_0$.
\end{theorem}

\begin{proof}
From the assumption, we have
$\frac{4}{\beta_n(1-\alpha_n)-\alpha_n}\leq\frac{4}{\beta(1-\alpha)-\alpha}<\infty$.
Let $^*C_n=\{z\in C: \|x_n-Tx_n\|\leq
\frac{4}{\beta_n(1-\alpha_n)-\alpha_n}\|x_n-z\|\}$, then using Lemma
6.4, $F(T)\subset C_n^*\subset C_n\subset {^*C_n}$. Hence, by Lemma
2.6 and Theorem 3.1, $x_n\rightarrow P_{F(T)}x_0$.
\end{proof}

\begin{theorem}
Let $C$ be a nonempty closed convex subset of a Hilbert space $H$
and $T$ a nonexpansive mapping of \ $C$ into itself such that
$F(T)\neq \O$. Suppose $x_0\in C$ chosen arbitrarily and $\{x_n\}$
is given by
$$
\left\{\begin{array}{l}v_n=(1-\alpha_n)x_n+\alpha_n Tx_n\\
y_n=(1-\beta_n)x_n+\beta_n Tv_n\\
C_n'=\{z\in C: \|z-v_n\|\leq\|z-x_n\|\}\\
C_n''=\{z\in C: \|z-y_n\|^2\leq\|z-x_n\|^2+\beta_n(\|v_n\|^2-\|x_n\|^2+2\langle x_n-v_n,z\rangle)\}\\
C_n=C_n'\cap C_n''\\
Q_n=\{z\in C: \langle z-x_n,x_n-x_0\rangle\geq 0\}\\
x_{n+1}=P_{C_n^*\cap Q_n}x_0
\end{array}\right.
$$
where $C_n^*$ is a closed convex set with $F(T)\subset C_n^*\subset
C_n$ and $\{\alpha_n\}$ and $\{\beta_n\}$ are chosen such that
$0<\alpha\leq\alpha_n\leq 1$ and $0\leq\beta_n\leq 1$. Then,
$\{x_n\}$ converges strongly to $P_{F(T)}x_0$.
\end{theorem}

\begin{proof}
It is obvious that $F(T)\subset C_n^*\subset C_n\subset C_n'$.
Hence, by Theorem 4.3, $x_n\rightarrow P_{F(T)}x_0$.
\end{proof}

\section{Generalized CQ algorithms for Halpern' iteration
process}

In this section, we give some algorithms for Halpern's iteration
process. To prove the main theorems, we need the following lemmas.

\begin{lemma}
Let $C$ be a nonempty closed convex subset of a real Hilbert space
$H$. Let $T:C\rightarrow C$ be a Lipschitz pseudo-contractive
mapping with Lipschitz constant $L\geq 1$. $\forall x,x_0\in C$ and
$\alpha\in [0,1]$, let
$$
y=(1-\alpha)x_0+\alpha Tx
$$
and
$$C_x=\{z\in C: \|y-z\|^2\leq\|x-z\|^2+2(1-\alpha)\langle
x-x_0,z\rangle+\theta\}
$$
where
$\theta=(1-\alpha)(\|x_0\|^2-\|x\|^2)+\alpha\|x-Tx\|^2-\alpha(1-\alpha)\|x_0-Tx\|^2$.
Then, there holds $C_x$ is a closed convex set with $F(T)\subset
C_x$.
\end{lemma}

\begin{proof}
Obviously, by Lemma 2.3, we can conclude $C_x$ is closed and convex.
Let $p\in F(T)$. We have,
$$
\aligned \|y-p\|^2&=\|(1-\alpha)(x_0-p)+\alpha(Tx-p)\|^2\\
&=(1-\alpha)\|x_0-p\|^2+\alpha\|Tx-p\|^2-\alpha(1-\alpha)\|x_0-Tx\|^2\\
&\leq(1-\alpha)\|x_0-p\|^2+\alpha(\|x-p\|^2+\|x-Tx\|^2)-\alpha(1-\alpha)\|x_0-Tx\|^2\\
&=\|x-p\|^2+(1-\alpha)(\|x_0-p\|^2-\|x-p\|^2)\\
&\quad+\alpha\|x-Tx\|^2-\alpha(1-\alpha)\|x_0-Tx\|^2\\
&=\|x-p\|^2+(1-\alpha)(\|x_0\|^2-\|x\|^2+2\langle x-x_0,p\rangle)\\
&\quad+\alpha\|x-Tx\|^2-\alpha(1-\alpha)\|x_0-Tx\|^2\\
&=\|x-p\|^2+2(1-\alpha)\langle
x-x_0,p\rangle+(1-\alpha)(\|x_0\|^2-\|x\|^2)\\
&\quad+\alpha\|x-Tx\|^2-\alpha(1-\alpha)\|x_0-Tx\|^2
\endaligned
$$
Let
$\theta=(1-\alpha)(\|x_0\|^2-\|x\|^2)+\alpha\|x-Tx\|^2-\alpha(1-\alpha)\|x_0-Tx\|^2$.
Then,
$$
\|y-p\|^2\leq\|x-p\|^2+2(1-\alpha)\langle x-x_0,p\rangle+\theta
$$
Therefore, $p\in C_x$, i.e., $F(T)\subset C_x$.
\end{proof}

\begin{lemma}
Let $C$ be a nonempty closed convex subset of a real Hilbert space
$H$. Let $T:C\rightarrow C$ be a $\kappa-$strict pseudo-contractive
mapping for some $0\leq\kappa<1$ with $F(T)\neq\O$. $\forall
x,x_0\in C$ and $\alpha\in [0,1]$, let
$$
y=(1-\alpha)x_0+\alpha Tx,
$$
$$
C_x=\{z\in C: \|y-z\|^2\leq\|x-z\|^2+2(1-\alpha)\langle
x-x_0,z\rangle+\theta\}
$$
and
$$
^*C_x=\{z\in C:
\|x-Tx\|\leq\frac{2}{1-\sqrt{\alpha\kappa}}\|x-z\|+\frac{\sqrt{1-\alpha}}{1-\sqrt{\alpha\kappa}}(\|x_0-Tx\|+\|x_0-z\|)\}
$$
where
$\theta=(1-\alpha)(\|x_0\|^2-\|x\|^2)+\alpha\kappa\|x-Tx\|^2-\alpha(1-\alpha)\|x_0-Tx\|^2$.
Then, there holds $C_x$ is a closed convex set with $F(T)\subset
C_x\subset {^*C_x}$.
\end{lemma}

\begin{proof}
Obviously, by Lemma 2.3, we can conclude $C_x$ is closed and convex.
Let $p\in F(T)$. We have,
$$
\aligned \|y-p\|^2&=\|(1-\alpha)(x_0-p)+\alpha(Tx-p)\|^2\\
&=(1-\alpha)\|x_0-p\|^2+\alpha\|Tx-p\|^2-\alpha(1-\alpha)\|x_0-Tx\|^2\\
&\leq(1-\alpha)\|x_0-p\|^2+\alpha(\|x-p\|^2+\kappa\|x-Tx\|^2)-\alpha(1-\alpha)\|x_0-Tx\|^2\\
&=\|x-p\|^2+(1-\alpha)(\|x_0-p\|^2-\|x-p\|^2)\\
&\quad+\alpha\kappa\|x-Tx\|^2-\alpha(1-\alpha)\|x_0-Tx\|^2\\
&=\|x-p\|^2+(1-\alpha)(\|x_0\|^2-\|x\|^2+2\langle x-x_0,p\rangle)\\
&\quad+\alpha\kappa\|x-Tx\|^2-\alpha(1-\alpha)\|x_0-Tx\|^2\\
&=\|x-p\|^2+2(1-\alpha)\langle
x-x_0,p\rangle+(1-\alpha)(\|x_0\|^2-\|x\|^2)\\
&\quad+\alpha\kappa\|x-Tx\|^2-\alpha(1-\alpha)\|x_0-Tx\|^2
\endaligned
$$
Let
$\theta=(1-\alpha)(\|x_0\|^2-\|x\|^2)+\alpha\kappa\|x-Tx\|^2-\alpha(1-\alpha)\|x_0-Tx\|^2$.
Then,
$$
\|y-p\|^2\leq\|x-p\|^2+2(1-\alpha)\langle x-x_0,p\rangle+\theta
$$
Therefore, $p\in C_x$, i.e., $F(T)\subset C_x$. Let $u\in C_x$, then
$\forall x\in C$
\begin{equation}
\aligned \|y-u\|^2&\leq\|x-u\|^2+2(1-\alpha)\langle
x-x_0,u\rangle+(1-\alpha)(\|x_0\|^2-\|x\|^2)\\
&\quad+\alpha\kappa\|x-Tx\|^2-\alpha(1-\alpha)\|x_0-Tx\|^2\\
&=(1-\alpha)\|x_0-u\|^2+\alpha(\|x-u\|^2+\kappa\|x-Tx\|^2)\\
&\quad-\alpha(1-\alpha)\|x_0-Tx\|^2\\
&\leq(1-\alpha)\|x_0-u\|^2+\alpha(\|x-u\|^2+\kappa\|x-Tx\|^2)\\
&\leq(1-\alpha)\|x_0-u\|^2+\alpha(\|x-u\|+\sqrt{\kappa}\|x-Tx\|)^2\\
&\leq[\sqrt{1-\alpha}\|x_0-u\|+\sqrt{\alpha}(\|x-u\|+\sqrt{\kappa}\|x-Tx\|)]^2.
\endaligned
\end{equation}
It follows that,
\begin{equation}
\|y-u\|
\leq\sqrt{1-\alpha}\|x_0-u\|+\sqrt{\alpha}\|x-u\|+\sqrt{\alpha\kappa}\|x-Tx\|.
\end{equation}
Besides,
\begin{equation}
\aligned \|x-Tx\|&\leq\|x-y\|+\|y-Tx\|\\
&\leq\|x-u\|+\|y-u\|+(1-\alpha)\|x_0-Tx\|\\
\endaligned
\end{equation}
Substitute (7.2) into (7.3) can yield
\begin{equation}
(1-\sqrt{\alpha\kappa})\|x-Tx\|\leq
2\|x-u\|+\sqrt{1-\alpha}(\|x_0-Tx\|+\|x_0-u\|).
\end{equation}
From the assumption of the coefficients, we have
\begin{equation}
\|x-Tx\|\leq\frac{2}{1-\sqrt{\alpha\kappa}}\|x-u\|+\frac{\sqrt{1-\alpha}}{1-\sqrt{\alpha\kappa}}(\|x_0-Tx\|+\|x_0-u\|)
\end{equation}
which implies $u\in {^*C_x}$. So, $F(T)\subset C_x\subset {^*C_x}$.
\end{proof}

\begin{lemma}
Let $C$ be a nonempty closed convex subset of a real Hilbert space
$H$. Let $T$ be a nonexpansive mapping of \ $C$ into itself with
$F(T)\neq\O$. $\forall x_0,x\in C$ and $0\leq\alpha\leq 1$, let
$$
y=(1-\alpha)x_0+\alpha Tx,
$$
$$
C_x=\{z\in C: \|z-y\|^2\leq\|z-x\|^2+(1-\alpha)(\|x_0\|^2+2\langle
x-x_0,z\rangle)\}
$$
and
$$
^*C_x=\{z\in C:
\|x-Tx\|\leq2\|x-z\|+\sqrt{1-\alpha}[4\|x_0\|+\|x_0-Tx\|+\|x-x_0\|+\|x_0-z\|]\}
$$
Then, $C_x$ is a closed convex subset with $F(T)\subset C_x\subset
{^*C_x}$.
\end{lemma}

\begin{proof}
By Lemma 2.3, we see that $C_x$ is closed and convex. For any $p\in
F(T)$, we have
$$
\aligned\|y-p\|^2&=\|(1-\alpha)(x_0-p)+\alpha(Tx-p)\|^2\\
&\leq(1-\alpha)\|x_0-p\|^2+\alpha\|x-p\|^2\\
&=\|x-p\|^2+(1-\alpha)(\|x_0-p\|^2-\|x-p\|^2)\\
&=\|x-p\|^2+(1-\alpha)(\|x_0\|^2+2\langle x-x_0,p\rangle).
\endaligned
$$
Hence, $F(T)\subset C_x$. Let $u\in C_x$, then $\forall x\in C$
\begin{equation}
\aligned \|y-u\|^2&\leq\|x-u\|^2+(1-\alpha)(\|x_0\|^2+2\langle
x-x_0,u\rangle)\\
&\leq\|x-u\|^2+(1-\alpha)(\|x_0\|^2+\|x-x_0+u\|^2)\\
&\leq\|x-u\|^2+(1-\alpha)[\|x_0\|+\|x-x_0+u\|]^2\\
&\leq[\|x-u\|+\sqrt{1-\alpha}(2\|x_0\|+\|x+u\|)]^2.
\endaligned
\end{equation}
From (7.6), we obtain
\begin{equation}
\|y-u\|\leq\|x-u\|+\sqrt{1-\alpha}[2\|x_0\|+\|x+u\|].
\end{equation}
We also have,
\begin{equation}
\aligned \|x-Tx\|&\leq\|x-u\|+\|y-u\|+\|y-Tx\|\\
&=\|x-u\|+\|y-u\|+(1-\alpha)\|x_0-Tx\|.
\endaligned
\end{equation}
Combining (7.7) and (7.8), we get
\begin{equation}
\aligned \|x-Tx\|&\leq\|x-u\|+(1-\alpha)\|x_0-Tx\|+\|y-u\|\\
&\leq\|x-u\|+(1-\alpha)\|x_0-Tx\|+\|x-u\|\\
&\quad+\sqrt{1-\alpha}[2\|x_0\|+\|x+u\|]\\
&\leq2\|x-u\|+\sqrt{1-\alpha}[4\|x_0\|+\|x_0-Tx\|+\|x-x_0\|+\|u-x_0\|].
\endaligned
\end{equation}
From (7.9), we can conclude $u\in {^*C_x}$. So, $F(T)\subset
C_x\subset {^*C_x}$.
\end{proof}

\begin{theorem}
Let $C$ be a nonempty closed convex subset of a real Hilbert space
$H$ and $T$ a Lipschitz pseudo-contraction from $C$ into itself with
the Lipschitz constant $L\geq 1$ and $F(T)\neq\O$. Assume sequence
$\{\tau_n\}\subset [\tau,1]$ with $\tau\in (0,1]$, sequence
$\{\alpha_n\}\subset [a,b]$ with $a,b\in (0,\frac{1}{L+1})$ and
sequence $\{\beta_n\}$ satisfies that $\beta_n\in [0,1]$. Let
$\{x_n\}$ be a sequence generated by the following manner:
\begin{equation}
\left\{\begin{array}{l}x_0\in C \ chosen \ arbitrarily\\
y_n=(1-\beta_n)x_0+\beta_n Tx_n\\
v_n=(1-\alpha_n) x_n+\alpha_n Tx_n\\
C_n'=\{z\in C: \tau_n\alpha_n[1-(1+L)\alpha_n]\|x_n-T x_n\|^2\leq\langle x_n-z,v_n-Tv_n\rangle\}\\
C_n''=\{z\in C: \|y_n-z\|^2\leq\|x_n-z\|^2+2(1-\beta_n)\langle
x_n-x_0,z\rangle+\theta_n\}\\
C_n=C_n'\cap C_n''\\
Q_n=\{z\in C: \langle z-x_n,x_n-x_0\rangle\geq 0\}\\
x_{n+1}=P_{C_n^*\cap Q_n}x_0
\end{array}\right.
\end{equation}
where $C_n^*$ is a closed convex set with $F(T)\subset C_n^*\subset
C_n$ and
$\theta_n=(1-\beta_n)(\|x_0\|^2-\|x_n\|^2)+\beta_n\|x_n-Tx_n\|^2-\beta_n(1-\beta_n)\|x_0-Tx_n\|^2$.
Then $\{x_n\}$ converges strongly to $P_{F(T)}x_0$.
\end{theorem}

\begin{proof}
It is obvious that $F(T)\subset C_n^*\subset C_n\subset C_n'$.
Hence, by Theorem 5.1, $x_n\rightarrow P_{F(T)}x_0$.
\end{proof}

\begin{theorem}
Let $C$ be a nonempty closed convex subset of a real Hilbert space
$H$. Let $T:C\rightarrow C$ be a $\kappa-$strict-pseudo-contraction
for some $0\leq\kappa<1$ such that $F(T)\neq\O$. Suppose that
$\{\alpha_n\}$ is a real sequence in $[0,1]$ satisfies that
$\lim_{n\rightarrow\infty}\alpha_n=1$. Let a sequence $\{x_n\}$ be
generated by
$$
\left\{\begin{array}{l}x_0\in C \ chosen \ arbitrarily\\
y_n=(1-\alpha_n)x_0+\alpha_n Tx_n\\
C_n=\{z\in C: \|y_n-z\|^2\leq\|x_n-z\|^2+2(1-\alpha_n)\langle
x_n-x_0,z\rangle+\theta_n\}\\
Q_n=\{z\in C:\langle z-x_n,x_n-x_0\rangle\geq 0\}\\
x_{n+1}=P_{C_n^*\cap Q_n}\\
\end{array}\right.
$$
where $C_n^*$ is a closed convex set with $F(T)\subset C_n^*\subset
C_n$ and
$\theta_n=(1-\alpha_n)(\|x_0\|^2-\|x_n\|^2)+\alpha_n\kappa\|x_n-Tx_n\|^2-\alpha_n(1-\alpha_n)\|x_0-Tx_n\|^2$.
Then, $\{x_n\}$ converges strongly to $P_{F(T)}x_0$.
\end{theorem}

\begin{proof}
From the assumption, we have
$\frac{2}{1-\sqrt{\alpha_n\kappa}}\leq\frac{2}{1-\sqrt{\kappa}}<\infty$.
Easily, we can prove $\|x_0-Tx_n\|+\|x_0-x_{n+1}\|$ is bounded.
Then,
$\lim_{n\rightarrow\infty}\frac{\sqrt{1-\alpha_n}}{1-\sqrt{\alpha_n\kappa}}(\|x_0-Tx_n\|+\|x_0-x_{n+1}\|)=0$.
Let $^*C_n=\{z\in C:
\|x_n-Tx_n\|\leq\frac{2}{1-\sqrt{\alpha_n\kappa}}\|x_n-z\|+\frac{\sqrt{1-\alpha_n}}{1-\sqrt{\alpha_n\kappa}}(\|x_0-Tx_n\|+\|x_0-z\|)\}$,
then using Lemma 7.2, $F(T)\subset C_n^*\subset C_n\subset {^*C_n}$.
Hence, by Lemma 2.6 and Theorem 3.1, $x_n\rightarrow P_{F(T)}x_0$.
\end{proof}

\begin{theorem}
Let $C$ be a nonempty closed convex subset of a Hilbert space $H$
and $T$ a $\kappa-$strict pseudo-contraction of $C$ into itself for
some $0\leq\kappa<1$ with $F(T)\neq\O$. Suppose $x_0\in C$ chosen
arbitrarily and $\{x_n\}$ is given by
\begin{equation}
\left\{\begin{array}{l}v_n=(1-\alpha_n) x_n+\alpha_n Tx_n\\
y_n=(1-\beta_n)x_0+\beta_n Tx_n\\
C_n'=\{z\in C: \|v_n-z\|^2\leq\|x_n-z\|^2+\alpha_n(\kappa-(1-\alpha_n))\|x_n-Tx_n\|^2\}\\
C_n''=\{z\in C: \|y_n-z\|^2\leq\|x_n-z\|^2+2(1-\beta_n)\langle
x_n-x_0,z\rangle+\theta_n\}\\
C_n=C_n'\cap C_n''\\
Q_n=\{z\in C: \langle z-x_n,x_n-x_0\rangle\geq 0\}\\
x_{n+1}=P_{C_n^*\cap Q_n}x_0
\end{array}\right.
\end{equation}
where $C_n^*$ is a closed convex set with $F(T)\subset C_n^*\subset
C_n$, $\{\alpha_n\}$ is chosen such that $0<\alpha\leq\alpha_n\leq
1$, $\{\beta_n\}$ is a sequence in $[0,1]$ and
$\theta_n=(1-\beta_n)(\|x_0\|^2-\|x_n\|^2)+\beta_n\kappa\|x_n-Tx_n\|^2-\beta_n(1-\beta_n)\|x_0-Tx_n\|^2$.
Then, $\{x_n\}$ converges strongly to $P_{F(T)}x_0$.
\end{theorem}

\begin{proof}
It is obvious that $F(T)\subset C_n^*\subset C_n\subset C_n'$.
Hence, by Theorem 5.3, $x_n\rightarrow P_{F(T)}x_0$.
\end{proof}

The following theorem is a deduced result of Theorem 7.3.

\begin{theorem}
Let $C$ be a nonempty closed convex subset of a real Hilbert space
$H$. Let $T:C\rightarrow C$ be a nonexpansive mapping with
$F(T)\neq\O$. Suppose that $\{\alpha_n\}$ is a real sequence in
$[0,1]$ satisfies that $\lim_{n\rightarrow\infty}\alpha_n=1$. Let a
sequence $\{x_n\}$ be generated by
$$
\left\{\begin{array}{l}x_0\in C \ chosen \ arbitrarily\\
y_n=(1-\alpha_n)x_0+\alpha_n Tx_n\\
C_n=\{z\in C: \|y_n-z\|^2\leq\|x_n-z\|^2+2(1-\alpha_n)\langle
x_n-x_0,z\rangle+\theta_n\}\\
Q_n=\{z\in C:\langle z-x_n,x_n-x_0\rangle\geq 0\}\\
x_{n+1}=P_{C_n^*\cap Q_n}\\
\end{array}\right.
$$
where $C_n^*$ is a closed convex set with $F(T)\subset C_n^*\subset
C_n$ and
$\theta_n=(1-\alpha_n)(\|x_0\|^2-\|x_n\|^2)-\alpha_n(1-\alpha_n)\|x_0-Tx_n\|^2$.
Then, $\{x_n\}$ converges strongly to $P_{F(T)}x_0$.
\end{theorem}

\begin{theorem}
Let $C$ be a nonempty closed convex subset of a Hilbert space $H$.
Let $T$ be a nonexpansive mapping of \ $C$ into itself such that
$F(T)\neq \O$. Assume that $\{\alpha_n\}$ is a sequences in $(0,1)$
such that $\lim_{n\rightarrow\infty}\alpha_n=1$. Define a sequence
$\{x_n\}$ in $C$ by algorithm:
$$
\left\{\begin{array}{l}x_0\in C \ chosen \ arbitrarily\\
y_n=(1-\alpha_n)x_0+\alpha_n Tx_n\\
C_n=\{z\in C:
\|z-y_n\|^2\leq\|z-x_n\|^2+(1-\alpha_n)(\|x_0\|^2+2\langle
x_n-x_0,z\rangle)\}\\
Q_n=\{z\in C: \langle z-x_n,x_n-x_0\rangle\geq 0\}\\
x_{n+1}=P_{C_n^*\cap Q_n}x_0
\end{array}\right.
$$
where $C_n^*$ is a closed convex set with $F(T)\subset C_n^*\subset
C_n$. Then, $\{x_n\}$ converges strongly to $P_{F(T)}x_0$.
\end{theorem}

\begin{proof}
Obviously, $4\|x_0\|+\|x_0-Tx_n\|+\|x_n-x_0\|+\|x_0-x_{n+1}\|$ is
bounded. Then,
$\lim_{n\rightarrow\infty}\sqrt{1-\alpha_n}[4\|x_0\|+\|x_0-Tx_n\|+\|x_n-x_0\|+\|x_0-x_{n+1}\|]=0$.
Let $^*C_n=\{z\in C: \|x_n-Tx_n\|\leq2\|x_n-z\|
+\sqrt{1-\alpha_n}[4\|x_0\|+\|x_0-Tx_n\|+\|x_n-x_0\|+\|x_0-z\|]\}$,
then using Lemma 7.3, $F(T)\subset C_n^*\subset C_n\subset {^*C_n}$.
Hence, by Lemma 2.6 and Theorem 3.1, $x_n\rightarrow P_{F(T)}x_0$.
\end{proof}

\begin{remark}
In this theorem, let $C_n^*=C_n$, then we obtain an algorithm which
is also proposed in \cite{9}. And Theorem 7.4 is also the deduced
result of Theorem 7.5.
\end{remark}

\begin{remark}
In last three sections, $C_n$ itself is closed and convex. So,
setting $C_n^*=C_n$, we can yield normal CQ algorithms.
\end{remark}

\section{Relations of different algorithms}

In \cite{12}, Takahashi, Takeuchi and Kubota obtained another strong
convergence theorem for nonexpansive mappings, named monotone C
method.

\begin{theorem}
Let $C$ be a nonempty closed convex subset of a Hilbert space $H$
and $T$ a nonexpansive mapping of \ $C$ into itself such that
$F(T)\neq \O$. Suppose $x_0\in C_0=C$ chosen arbitrarily and
$\{x_n\}$ is given by
\begin{equation}
\left\{\begin{array}{l}y_n=(1-\alpha_n) x_n+\alpha_n Tx_n\\
C_{n+1}=\{z\in C_n: \|z-y_n\|\leq\|z-x_n\|\}\\
x_{n+1}=P_{C_{n+1}}x_0
\end{array}\right.
\end{equation}
where $\{\alpha_n\}$ is chosen such that $0<\alpha\leq\alpha_n\leq
1$. Then, $\{x_n\}$ converges strongly to $P_{F(T)}x_0$.
\end{theorem}

In \cite{13}, Su and Qin got a new hybrid method, named Monotone CQ
iteration processes.

\begin{theorem}
Let $C$ be a nonempty closed convex subset of a Hilbert space $H$
and $T$ a nonexpansive mapping of \ $C$ into itself such that
$F(T)\neq \O$. Suppose $x_0\in C_0=C$ chosen arbitrarily and
$\{x_n\}$ is given by
\begin{equation}
\left\{\begin{array}{l}y_n=(1-\alpha_n) x_n+\alpha_n Tx_n\\
C_0=\{z\in C: \|z-y_0\|\leq\|z-x_0\|\}\\
Q_0=C\\
C_n=\{z\in C_{n-1}\cap Q_{n-1}: \|z-y_n\|\leq\|z-x_n\|\}\\
Q_n=\{z\in C_{n-1}\cap Q_{n-1}: \langle z-x_n,x_n-x_0\rangle\geq 0\}\\
x_{n+1}=P_{C_n\cap Q_n}x_0
\end{array}\right.
\end{equation}
where $\{\alpha_n\}$ is chosen such that $0<\alpha\leq\alpha_n\leq
1$. Then, $\{x_n\}$ converges strongly to $P_{F(T)}x_0$.
\end{theorem}

\begin{remark}
Theorem 1.1, 8.1 and 8.2 seem different from each other. However,
the steps of their proof are more or less the same. So, they may
share some properties or may have some relations.
\end{remark}

In this section, we give the relations of the following four
theorems.

\begin{theorem}
Let $C$ be a nonempty closed convex subset of a Hilbert space $H$
and $T$ a nonexpansive mapping of \ $C$ into itself such that
$F(T)\neq \O$. Suppose $x_0\in C$ chosen arbitrarily and $\{x_n\}$
is given by
$$
\left\{\begin{array}{l}y_n=(1-\alpha_n) x_n+\alpha_n Tx_n\\
C_n=\{z\in C:\|y_n-z\|\leq\|x_n-z\|\}\\
Q_n=\{z\in C: \langle z-x_n,x_n-x_0\rangle\geq 0\}\\
x_{n+1}=P_{C_n(thm3)\cap Q_n}x_0
\end{array}\right.
$$
where $C_n(thm3)$ is a closed convex set with $F(T)\subset
C_n(thm3)\subset C_n$ and $\{\alpha_n\}$ is chosen such that
$0<\alpha\leq\alpha_n\leq 1$. Then, $\{x_n\}$ converges strongly to
$P_{F(T)}x_0$.
\end{theorem}

\begin{theorem}
Let $C$ be a nonempty closed convex subset of a Hilbert space $H$
and $T$ a nonexpansive mapping of \ $C$ into itself such that
$F(T)\neq \O$. Suppose $x_0\in C$ chosen arbitrarily and $\{x_n\}$
is given by
$$
\left\{\begin{array}{l}y_n=(1-\alpha_n) x_n+\alpha_n Tx_n\\
C_n=\{z\in C:\|y_n-z\|\leq\|x_n-z\|\}\\
Q_0=C\\
Q_n=\{z\in Q_{n-1}: \langle z-x_n,x_n-x_0\rangle\geq 0\}\\
x_{n+1}=P_{C_n(thm4)\cap Q_n}x_0
\end{array}\right.
$$
where $C_n(thm4)$ is a closed convex set with $F(T)\subset
C_n(thm4)\subset C_n$ and $\{\alpha_n\}$ is chosen such that
$0<\alpha\leq\alpha_n\leq 1$. Then, $\{x_n\}$ converges strongly to
$P_{F(T)}x_0$.
\end{theorem}

\begin{theorem}
Let $C$ be a nonempty closed convex subset of a Hilbert space $H$
and $T$ a nonexpansive mapping of \ $C$ into itself such that
$F(T)\neq \O$. Suppose $x_0\in C_0=C$ chosen arbitrarily and
$\{x_n\}$ is given by
$$
\left\{\begin{array}{l}y_n=(1-\alpha_n) x_n+\alpha_n Tx_n\\
C_{n+1}=\{z\in C_n(thm5):\|y_n-z\|\leq\|x_n-z\|\}\\
x_{n+1}=P_{C_{n+1}(thm5)}x_0
\end{array}\right.
$$
where $C_{n+1}(thm5)$ is a closed convex set with $F(T)\subset
C_{n+1}(thm5)\subset C_{n+1}$ and $\{\alpha_n\}$ is chosen such that
$0<\alpha\leq\alpha_n\leq 1$. Then, $\{x_n\}$ converges strongly to
$P_{F(T)}x_0$.
\end{theorem}

\begin{theorem}
Let $C$ be a nonempty closed convex subset of a Hilbert space $H$
and $T$ a nonexpansive mapping of \ $C$ into itself such that
$F(T)\neq \O$. Suppose $x_0\in C$ chosen arbitrarily and $\{x_n\}$
is given by
$$
\left\{\begin{array}{l}y_n=(1-\alpha_n) x_n+\alpha_n Tx_n\\
C_0=\{z\in C:\|y_0-z\|\leq\|x_0-z\|\}\\
Q_0=C\\
C_n=\{z\in C_{n-1}(thm6)\cap Q_{n-1}:\|y_n-z\|\leq\|x_n-z\|\}\\
Q_n=\{z\in C_{n-1}(thm6)\cap Q_{n-1}: \langle z-x_n,x_n-x_0\rangle\geq 0\}\\
x_{n+1}=P_{C_n(thm6)\cap Q_n}x_0
\end{array}\right.
$$
where $C_n(thm6)$ is a closed convex set with $F(T)\subset
C_n(thm6)\subset C_n$ and $\{\alpha_n\}$ is chosen such that
$0<\alpha\leq\alpha_n\leq 1$. Then, $\{x_n\}$ converges strongly to
$P_{F(T)}x_0$.
\end{theorem}

\begin{proposition}
Theorem 8.3 TRUE $\Rightarrow$ Theorem 8.4 TRUE $\Rightarrow$
Theorem 8.5 TRUE $\Leftrightarrow$ Theorem 8.6 TRUE, Where Theorem
8.3 TRUE indicates that Theorem 8.3 is valid.
\end{proposition}

\begin{proof}
Clearly, Theorem 8.3 is valid.

(1) Theorem 8.3 TRUE $\Rightarrow$ Theorem 8.4 TRUE. Obviously,
$F(T)\subset C_n\cap(\bigcap_{i=0}^{n-1} Q_i)$. So, there exists a
closed convex set $C_n(thm8)$ such that $F(T)\subset
C_n(thm8)\subset C_n\cap(\bigcap_{i=0}^{n-1} Q_i)\subset C_n$. By
Theorem 8.3, we obtain the following theorem.
\begin{theorem}
Let $C$ be a nonempty closed convex subset of a Hilbert space $H$
and $T$ a nonexpansive mapping of \ $C$ into itself such that
$F(T)\neq \O$. Suppose $x_0\in C$ chosen arbitrarily and $\{x_n\}$
is given by
$$
\left\{\begin{array}{l}y_n=(1-\alpha_n) x_n+\alpha_n Tx_n\\
C_n=\{z\in C:\|y_n-z\|\leq\|x_n-z\|\}\\
Q_n=\{z\in C: \langle z-x_n,x_n-x_0\rangle\geq 0\}\\
x_{n+1}=P_{C_n(thm8)\cap(\bigcap_{i=0}^{n} Q_i)}x_0
\end{array}\right.
$$
where $C_n(thm8)$ is a closed convex set with $F(T)\subset
C_n(thm8)\subset C_n$ and $\{\alpha_n\}$ is chosen such that
$0<\alpha\leq\alpha_n\leq 1$. Then, $\{x_n\}$ converges strongly to
$P_{F(T)}x_0$.
\end{theorem}

Let $C_n(thm8)=C_n(thm4)$, then, Theorem 8.8 is equivalent to
Theorem 8.4.

(2) Theorem 8.4 TRUE $\Rightarrow$ Theorem 8.5 TRUE. Actually,
Theorem 8.5 can be rewritten as

\begin{theorem}
Let $C$ be a nonempty closed convex subset of a Hilbert space $H$
and $T$ a nonexpansive mapping of \ $C$ into itself such that
$F(T)\neq \O$. Suppose $x_0\in C$ chosen arbitrarily and $\{x_n\}$
is given by
$$
\left\{\begin{array}{l}y_n=(1-\alpha_n) x_n+\alpha_n Tx_n\\
C_0=\{z\in C:\|y_0-z\|\leq\|x_0-z\|\}\\
C_n=\{z\in C_{n-1}(thm9):\|y_n-z\|\leq\|x_n-z\|\}\\
x_{n+1}=P_{C_n(thm9)}x_0
\end{array}\right.
$$
where $C_n(thm9)$ is a closed convex set with $F(T)\subset
C_n(thm9)\subset C_n$ and $\{\alpha_n\}$ is chosen such that
$0<\alpha\leq\alpha_n\leq 1$. Then, $\{x_n\}$ converges strongly to
$P_{F(T)}x_0$.
\end{theorem}

Since $x_n=P_{C_{n-1}(thm9)}x_0$, then, $C_{n-1}(thm9)\subset
Q_n=\{z\in C: \langle z-x_n,x_n-x_0\rangle\geq 0\}$. Together with
$C_n(thm9)\subset C_{n-1}(thm9)$, we claim that
$C_n(thm9)\subset(\bigcap_{i=0}^n Q_i)$, i.e.,
$C_n(thm9)\cap(\bigcap_{i=1}^n Q_i)=C_n(thm9)$, where $Q_i=\{z\in C:
\langle z-x_i,x_i-x_0\rangle\geq 0\}$. Hence, Theorem 8.9 is
equivalent to

\begin{theorem}
Let $C$ be a nonempty closed convex subset of a Hilbert space $H$
and $T$ a nonexpansive mapping of \ $C$ into itself such that
$F(T)\neq \O$. Suppose $x_0\in C$ chosen arbitrarily and $\{x_n\}$
is given by
$$
\left\{\begin{array}{l}y_n=(1-\alpha_n) x_n+\alpha_n Tx_n\\
C_0=\{z\in C:\|y_0-z\|\leq\|x_0-z\|\}\\
C_n=\{z\in C_{n-1}(thm10):\|y_n-z\|\leq\|x_n-z\|\}\\
Q_n=\{z\in C: \langle z-x_n,x_n-x_0\rangle\geq 0\}\\
x_{n+1}=P_{C_n(thm10)\cap(\bigcap_{i=0}^n Q_i)}x_0
\end{array}\right.
$$
where $C_n(thm10)$ is a closed convex set with $F(T)\subset
C_n(thm10)\subset C_n$ and $\{\alpha_n\}$ is chosen such that
$0<\alpha\leq\alpha_n\leq 1$. Then, $\{x_n\}$ converges strongly to
$P_{F(T)}x_0$.
\end{theorem}

It is easy to observe that Theorem 8.10 is equal to

\begin{theorem}
Let $C$ be a nonempty closed convex subset of a Hilbert space $H$
and $T$ a nonexpansive mapping of \ $C$ into itself such that
$F(T)\neq \O$. Suppose $x_0\in C$ chosen arbitrarily and $\{x_n\}$
is given by
$$
\left\{\begin{array}{l}y_n=(1-\alpha_n) x_n+\alpha_n Tx_n\\
C_0=\{z\in C:\|y_0-z\|\leq\|x_0-z\|\}\\
Q_0=C\\
C_n=\{z\in C_{n-1}(thm11):\|y_n-z\|\leq\|x_n-z\|\}\\
Q_n=\{z\in Q_{n-1}: \langle z-x_n,x_n-x_0\rangle\geq 0\}\\
x_{n+1}=P_{C_n(thm11)\cap Q_n}x_0
\end{array}\right.
$$
where $C_n(thm11)$ is a closed convex set with $F(T)\subset
C_n(thm11)\subset C_n$ and $\{\alpha_n\}$ is chosen such that
$0<\alpha\leq\alpha_n\leq 1$. Then, $\{x_n\}$ converges strongly to
$P_{F(T)}x_0$.
\end{theorem}

Clearly, Theorem 8.11 can be deduced by Theorem 8.4. So, Theorem 8.4
TRUE $\Rightarrow$ Theorem 8.5 TRUE.

(3) Theorem 8.5 TRUE $\Leftrightarrow$ Theorem 8.6 TRUE. Obviously,
let $C_n(thm11)=C_n(thm6)$, $C_n(thm11)\cap Q_n$ in Theorem 8.11 is
equal to $C_n(thm6)\cap Q_n$ in Theorem 8.6. Hence, Theorem 8.11 is
equivalent to Theorem 8.6.

\end{proof}

Moreover, if we take $C_n(thm5)=C_n$ and $C_n(thm6)=C_n$, then, we
can conclude that Theorem 8.1 is equivalent to Theorem 8.2.

\begin{remark}
From the proof of the Proposition 8.7, we observe that the
proposition is independent of mapping $T$. If Theorem 8.3, 8.4, 8.5
and 8.6 represent CQ method, monotone Q method, monotone C method
and monotone CQ method, respectively, then, we have the following
relations:

CQ method TRUE $\Rightarrow$ monotone Q method TRUE $\Rightarrow$
monotone C method TRUE $\Leftrightarrow$ monotone CQ method TRUE.
\end{remark}


\begin{thebibliography}{00}

\bibitem{2}
B. Halpern, Fixed points of nonexpanding maps, Bull. Amer. Math.
Soc. 73:957-961(1967).
\bibitem{4}
W.R. Mann, Mean value methods in iteration, Proc. Amer. Math. Soc.
4:506-510(1953).
\bibitem{3}
S. Ishikawa, Fixed point by a new iteration method, Proc. Am. Math.
Soc. 44:147-150(1974).
\bibitem{1}
A. Genel, J. Lindenstrass, An example concerning fixed points,
Israel J. Math. 22:81-86(1975).
\bibitem{5}
G. Marino, H.K. Xu. Weak and strong convergence theorems for strict
pseudo-contractions in Hilbert spaces. J. Math. Anal. Appl.
329:336-346(2007).
\bibitem{6}
K. Nakajo, W. Takahashi. Strong convergence theorems for
nonexpansive mappings and nonexpansive semigroups, J. Math. Anal.
Appl. 279:372-379(2003).
\bibitem{8}
X.L. Qin, Y.J. Cho, S.M. Kang, H.Y. Zhou, Convergence theorems of
common fixed points for a family of Lipschitz
quasi-pseudocontractions, Nonlinear Analysis, 71:685-690(2009).
\bibitem{9}
C.M. Yanes, H.K. Xu, Strong convergence of the CQ method for fixed
point iteration processes. Nonlinear Analysis. 64:2400-2411(2006).
\bibitem{10}
Y.H. Yao, Y.C Liou, G. Marino, A hybrid algorithm for
pseudo-contractive mappings, Nonlinear Analysis, 71:4997-5002(2009).
\bibitem{11}
H.Y. Zhou, Convergence theorems of fixed points for Lipschitz
pseudo-contractions in Hilbert spaces, J. Math. Anal. Appl.
343:546-556(2008).
\bibitem{7}
Z. Opial, Weak convergence of the sequence of successive
approximations for nonexpansive mappings, Bull. Amer. Math. Soc.
73:591-597(1967).
\bibitem{12}
W. Takahashi, Y. Takeuchi, R. Kubota. Strong convergence theorems by
hybrid methods for families of nonexpansive mappings in Hilbert
spaces, J. Math. Anal. Appl. 341:276-286(2008).
\bibitem{13}
Y.F. Su, X.L. Qin, Monotone CQ iteration processes for nonexpansive
semigroups and maximal monotone operators. Nonlinear Anal.
68:3657-3664(2008).


\end{thebibliography}
\end{document}